\documentclass[11pt]{article}
\usepackage{amsmath}
\usepackage{dsfont}
\usepackage{cmap}
\usepackage{indentfirst}
\usepackage{amsmath,amssymb,amsthm,graphicx,mathrsfs}
\usepackage[numbers,sort&compress]{natbib}
\usepackage{color}
\numberwithin{equation}{section}
\newcommand{\beq}{\begin{equation}}
	\newcommand{\enq}{\end{equation}}

\newtheorem{Theorem}{Theorem}[section]
\newtheorem{Lemma}[Theorem]{Lemma}

\newtheorem{Definition}[Theorem]{Definition}
\newtheorem{Remark}[subsection]{Remark}

\newcommand{\benu}{\begin{enumerate}}
	\newcommand{\beqa}{\begin{eqnarray}}
		\newcommand{\beqan}{\begin{eqnarray*}}
			\newcommand{\eay}{\end{array}}
		\newcommand{\edm}{\end{displaymath}}
	\newcommand{\eenu}{\end{enumerate}}
\newcommand{\eeq}{\end{equation}}
\newcommand{\eeqa}{\end{eqnarray}}
\newcommand{\eeqan}{\end{eqnarray*}}

\newcommand{\br}{\begin{Remark}}
\newcommand{\er}{\end{Remark}}

\newcommand{\bqa}{\begin{eqnarray}}
\newcommand{\eqa}{\end{eqnarray}}
\newcommand{\bqw}{\begin{eqnarray*}}
\newcommand{\eqw}{\end{eqnarray*}}

\newcommand{\bea}{\begin{array}{cc}}
\newcommand{\ena}{\end{array}}

\begin{document}
\begin{center}
{\large \bf
Regularity of pullback attractors for nonclassical diffusion equations with delay}\\
\vspace{0.20in}Yuming Qin$^{1,\ast}$ $\ $ Qitao Cai$^{2}$ $\ $ Ming Mei$^{3}$ $\ $ Ke Wang$^{4}$\\
\end{center}
$^{1}$ Department of  Mathematics, Institute for Nonlinear Sciences, Donghua University, Shanghai 201620, P. R. China.\\
E-mails: yuming$\_$qin@hotmail.com, yuming@dhu.edu.cn\\
$^{2}$ Department of  Mathematics,  Donghua University, Shanghai 201620, P. R. China.\\
E-mails: mrccc2017@163.com\\
$^{3}$ Department of Mathematics, Champlain College Saint-Lambert, Quebec, J4P 3P2, Canada.\\
Department of Mathematics and Statistics, McGill University, Montreal, Quebec, H3A 2K6, Canada.\\
E-mails: ming.mei@mcgill.ca\\
$^{4}$ Department of  Mathematics, Institute for Nonlinear Sciences, Donghua University, Shanghai 201620, P. R. China.\\
E-mails: kwang@dhu.edu.cn
\vspace{3mm}

\begin{abstract}
In this paper, we mainly study the regularity of pullback $\mathcal{D}$-attractors for a nonautonomous nonclassical diffusion equation with delay term $b(t,u_t)$ which contains some hereditary characteristics. Under a critical nonlinearity $f$, a time-dependent force $g(t,x)$ with exponential growth and a delayed force term $b(t,u_t)$, we prove that there exists a pullback $\mathcal{D}$-attractor $\mathcal{A}=\{A(t):t \in \mathbb{R}\}$ in $\mathbb{K}^1=H_0^1(\Omega) \times L^2((-h,0);L^2(\Omega))$ to problem \eqref{ine01} and for each $t \in \mathbb{R}$, $A(t)$ is bounded in $\mathbb{K}^2=H^2(\Omega) \cap H_0^1(\Omega) \times L^2((-h,0);L^2(\Omega))$.
\end{abstract}
\vspace{3mm} \hspace{4mm}{\bf Key words:} Nonclassical diffusion equation, regularity, pullback attractors, delay \\
\vspace{3mm}\hspace{2mm} {\bf AMS Subject Classifications}: 35K57; 35B65; 35B41; 34K99;

\section{Introduction}
\setcounter{equation}{0}
\let\thefootnote\relax\footnote{*Corresponding author. E-mails: yuming$\_$qin@hotmail.com, yuming@dhu.edu.cn}

	In  this paper, we consider the following nonclassical diffusion equation with delay:
\begin{equation}\label{ine01} 
	\begin{cases}
		\partial_t u(t,x)- \Delta \partial_t u(t,x) -\Delta u(t,x)+f(u(t,x))=b(t,u_t)(x)+g(t,x) & in \  (\tau,\infty)\times\Omega,\\
		u=0 & on \ (\tau,\infty)\times\partial\Omega,\\
		u(\tau,x)=u^0(x)&\tau \in \mathbb{R},x \in \Omega,\\
		u(\tau+\theta,x)=\phi(\theta,x) & \theta \in(-h,0),x \in \Omega,
	\end{cases}
\end{equation}
where $\Omega \subset \mathbb{R}^3$ is a bounded domain with smooth boundary $\partial \Omega $. $\tau \in \mathbb{R}$, $u^0 \in H_0^1(\Omega)$ is the initial condition at $\tau$ and $\phi \in L^2(\left[-h,0\right];L^2(\Omega))$ is also the initial condition in $\left[\tau-h,\tau\right]$, $h(>0)$ is the length of the delay effects. Moreover, for each $t \geq \tau$, we denote by $u_t$ the function defined in $\left[-h,0\right]$ by $u_t(\theta)=u(t+\theta), \,\theta \in \left[-h,0\right]$. The nonlinearity $f$ and the external forces $b$ and $g$ meet certain criteria later.

We assume the following conditions to hold (cf. \cite{hhmm1,wzl}):\\
$\left(\mathbf{H1}\right)$ The nonlinear term  $f\in C^1(\mathbb{R},\mathbb{R})$ with $f(0)=0$, satisfies the following growth and dissipation conditions:
\begin{equation}\label{1eq01}
	\left| f(u) - f(v)\right| \leq C \left| u-v \right| (1+\left|u\right|^4 + \left| v\right|^4),
\end{equation}
\begin{equation}\label{1eq02}
		\lim_{\left| u \right| \to \infty} \inf  \frac{f(u)}{u} > -\lambda_1. 
\end{equation}

According to \eqref{1eq01} and \eqref{1eq02}, it can be inferred that $f(u)$ satisfies the following inequalities: 
\begin{equation}
	f(u)u \geq-\mu u^2-C_1,
\end{equation}
\begin{equation}
	f'(u)  \geq -l,
\end{equation}
\begin{equation}\label{1eq07}
	\left| f(u) \right| \leq C_2(1+\left| u \right|^5),
\end{equation}
\begin{equation} \label{1eq03}
	\lim_{\left| u \right| \to \infty} \inf  \frac{uf(u)-kF(u)}{u^2} \geq0,
\end{equation}
\begin{equation} \label{1eq04}
	\lim_{\left| u \right| \to \infty} \inf  \frac{F(u)}{u^2} \geq 0,
\end{equation}
where $\mu$, $l$, $k$, $C_1$, $C_2$ are positive constants, $\lambda_1 >0$ is the first eigenvalue of $-\Delta$ in $\Omega$ with the homogeneous Dirichlet condition such that $\lambda_1 > \max \left\{ \mu, l \right\}$.

Let us denote by
\begin{equation*}
	F(u):=\int_{0}^{u} f(s) {\rm d}s.
\end{equation*}

We infer from equation \eqref{1eq03} and \eqref{1eq04} that for any $ \delta >0$, there exist positive constants $C_{\delta 1},C_{\delta 2}$ such that
\begin{equation} \label{1eq05}
	uf(u)-kF(u)+\delta u^2 + C_{\delta 1} \geq 0 \quad \forall u \in \mathbb{R},
\end{equation}
\begin{equation}\label{1eq06}
	F(u)+ \delta u^2 +C_{\delta 2} \geq 0 \quad \forall u \in \mathbb{R}.
\end{equation}
$\left(\mathbf{H2}\right)$ 	In terms of the time-dependent external force term without delay, we assume that $L^2_{loc}(\mathbb{R}; \mathbb{X})$ is the space of functions $g(s)$, $s\in \mathbb{R}$ with values in $\mathbb{X}$ and locally square integrable in time, that is,
\begin{equation*}
	\int_{a}^{b} \left\| g(s) \right\|^2_{\mathbb{X}} ds < \infty, \quad \forall \left[a,b\right] \subset \mathbb{R}.
\end{equation*}
$\left(\mathbf{H3}\right)$  We assume that the operator $b(\cdot,\cdot)$ is well-defined as $b: \mathbb{R} \times L^2((-h,0);L^2(\Omega)) \to L^2(\Omega)$ which is a time-dependent external force with delay, and it satisfies:
\begin{flushleft}
	(\uppercase \expandafter{\romannumeral 1}) for all $\phi \in L^2((-h,0);L^2(\Omega))$, the function $t \in \mathbb{R} \mapsto b(t,\phi) \in L^2(\Omega)$ is measurable;\\
	(\uppercase \expandafter{\romannumeral 2}) $b(t,0)=0$ for all $t \in \mathbb{R}$;\\
	(\uppercase \expandafter{\romannumeral 3}) there exists a constant $L_b >0$ such that for all $t \in \mathbb{R}$ and $\phi_1, \phi_2 \in L^2((-h,0);L^2(\Omega))$,
	\begin{equation*}
		\left \| b(t,\phi_1)-b(t,\phi_2) \right \| \leq L_b \left \| \phi_1 - \phi_2 \right \|_{L^2((-h,0);L^2(\Omega))};
	\end{equation*}
	(\uppercase \expandafter{\romannumeral 4}) there exists a constant $C_b>0$ such that for all $t \geq \tau$, and all $u,v \in L^2(\left[\tau-h,t\right];L^2(\Omega))$,
	\begin{equation*}
		\int_{\tau}^{t} \left \| b(s,u_s)-b(s,v_s) \right \|^2  {\rm d}s \leq C_b \int_{\tau -h}^{t} \left \|  u(s)-v(s) \right \|^2 {\rm d}s;
	\end{equation*}
	(\uppercase \expandafter{\romannumeral 5}) for all $t \geq \tau$, and all $u,v \in L^2(\left[\tau-h,t\right];L^2(\Omega))$, $b(t,u_t)$ also satisfies
	\begin{equation*}
		\int_{\tau}^{t} e^{\sigma s} \left \| b(s,u_s)-b(s,v_s) \right \|^2  {\rm d}s \leq C_b \int_{\tau -h}^{t} e^{\sigma s} \left \|  u(s)-v(s) \right \|^2 {\rm d}s.
	\end{equation*}
\end{flushleft}
\begin{Remark}
	From above assumptions, for $T> \tau$, the function $t \in \mathbb{R} \mapsto b(t,\phi) \in L^2(\Omega)$ is measurable and belongs to $L^{\infty}((\tau ,T);L^2(\Omega))$.
\end{Remark}

Many scholars have studied the nonclassical reaction diffusion equation, which was developed as a model to describe physical phenomena such as fluid mechanics, solid mechanics, non-Newtonian flows, and heat conduction theory (\cite{aec,cpj}). This type of equation has been widely used in physics, chemistry, biology, and other fields in recent years. Meanwhile, the long-time behavior and well-posedness of the solutions of the nonclassical reaction diffusion equation have been investigated (\cite{at,wpw,ctmm,scy1,ws,zkx1}, etc.). 

It is worth highlighting that the existence and regularity of attractors of related dynamic equations have been examined by numerous researcher under distinct settings. The authors in \cite{whq,zwg,mwx} have established the regularity of global attractors for the nonclassical diffusion equations. With reference to the pullback attractors, the existence as well as regularity for a wide class of non-autonomous, fractional, nonclassical diffusion equations with $(p,q)$-growth nonlinearities defined on unbounded domains have been obatained in \cite{wlw}. The regularity for a nonautonomous nonclassical diffusion equation with critical nonlinearity have been proved in \cite{wzl} and the existence pullback attractors for a nonautonomous nonclassical diffusion equation with delay have been proved in \cite{hhmm1}. But, the regularity of pullback attractors for a nonautonomous nonclassical diffusion equation with delay was not studied in \cite{hhmm1} and \cite{wzl}. Thus it is natural to study this problem in this paper.

The main aim in this paper is to prove the regularity of pullback $\mathcal{D}$-attractors $\mathcal{A}=\{A(t):t \in \mathbb{R}\}$ for problem \eqref{ine01}.  Particularly, we will prove that there exists a pullback $\mathcal{D}$-attractors $\mathcal{A}=\{A(t):t \in \mathbb{R}\}$ in $\mathbb{K}^1=H_0^1(\Omega) \times L^2((-h,0);L^2(\Omega))$ by weaks solutions of problem \eqref{ine01}. More specifically, $A(t)$ is bounded in $\mathbb{K}^2=H^2(\Omega) \cap H_0^1(\Omega) \times L^2((-h,0);L^2(\Omega))$ for each $t \in \mathbb{R}$. In order to obatin the regularity of pullback arrtactors for problem \eqref{ine01}, by extending the methods in \cite{wzl,cp,cm,fgmz,ze}, we split the solution of the problem into two parts: the first part satisfies exponentially decay by employing priori estimation techniques, and the second part satisfies suitable asymptotic behavior in some phase space with higher regularity by giving the operator $-\Delta$ a proper fractional power. Because of the term $-\Delta \partial_t u$, we know that if the intial datas belong to $H_0^1(\Omega)$, the solution for problem \eqref{ine01} is always in $H_0^1(\Omega)$ and has no higher regularity, which is similar to the hyperbolic equation (\cite{hhmm1,wzl}). Furthermore, the time-delay term forces us to consider the past history of the solution. To overcome these difficulties casued by the term $-\Delta \partial_t u$ and time-delay term, by generalizing the methods in \cite{hhmm1} and \cite{wzl}, we verify Theorem \ref{Th42}(see below).

The structure of this paper is the following. In Section \ref{sec2}, we set some notations and recall some definitions and abstract results on pullback $\mathcal{D}$-attractors and regularity (see \cite{hhmm1,lz,wzl}). In Section \ref{sec3}, we prove the existence of pullback $\mathcal{D}$-attractors in $H_0^1(\Omega) \times L^2((-h,0);L^2(\Omega))$ to problem \eqref{ine01}. Finally, throughout Section \ref{sec4}, we obtain the regularity of pullback $\mathcal{D}$-attractors. 

\section{Preliminaries}\label{sec2}
	Throughout this section, we  recall some fundamental concepts and results that will be the ingredients of this work.
	
	For convenience, throughout this paper, let $\left| u \right|$ be the modular (or absolute value) of $u$. Let $\left \| \cdot \right \|$ and $\langle \cdot,\cdot \rangle$ denote the norm and inner product in $L^2(\Omega)$, respectively; in the same way, let $\left \| \nabla \cdot \right \|$ and $ \langle \nabla \cdot,\nabla \cdot \rangle$ denote the norm and inner product in $H_0^1(\Omega)$, respectively. The norm in the Banach space $\mathbb{B}$  will be denoted by $\left \| \cdot \right \|_\mathbb{B}$. $C$ means any generic positive constant, which may be different from line to line and even in the same line. Let $\mathbb{X}$ be a complete metric space with metric $d$. Denote by $\mathcal{P} (\mathbb{X})$ the famliy of all non-empty subsets of $\mathbb{X}$, and suppose $\mathcal{D}$ is a non-empty class of parameterized sets $\widehat{D}=\{ D(t):t \in \mathbb{R}\} \subset \mathcal{P} (\mathbb{X})$.  For any $R>0$, $\overline{\mathbf{B}}_{\mathbb{Y}} \left(0,R\right)$ denotes the closed ball in $\mathbb{Y}$ centered at $0$ with radius $R$. In particular, denote by $\overline{\mathbf{B}}_{\mathbb{Y}} \left(R\right)=\overline{\mathbf{B}}_{\mathbb{Y}} \left(0,R\right)$ when the center is 0.

	Let $A:= -\Delta$ wiht the domain $D(A)=H^2(\Omega) \cap H^1_0(\Omega)$. We consider a family of (compactly) nested Hilbert spaces 
	\begin{equation*}
	\boldsymbol{E}^{\kappa}:=D(A^{\kappa /2}), \kappa \in \mathbb{R}, 
	\end{equation*}
	with inner products and norms given by 
	\begin{equation*}
	\langle u,v \rangle_{\kappa} = \langle A^{\kappa/2}u, A^{\kappa/2}v \rangle \  \text{and} \ \left\| u \right\|_{\kappa} =\| A^{\kappa/2}u \|.
	\end{equation*}
	In particular, $\boldsymbol{E}^0 =L^2(\Omega)$, $\boldsymbol{E}^1= H^1_0(\Omega)$ and $\boldsymbol{E}^2=H^2(\Omega) \cap H^1_0(\Omega)$.
	
	Therefore, we have the continuous embeddings $\boldsymbol{E}^{\kappa_1} \hookrightarrow \boldsymbol{E}^{\kappa_2}$ for any $\kappa_1>\kappa_2$,
	\begin{equation*}
	\boldsymbol{E}^{\kappa_1} \hookrightarrow L^{\frac{6}{3-2\kappa_1}}(\Omega), \quad \forall \kappa_1 \in \left[0,\frac{3}{2}\right),
	\end{equation*}
	and the following inequalities hold true:\\
	\textbf{Interpolation inequalities}: if $r=\nu s+(1-\nu)q$, where $r,s,q \in \mathbb{R}$, $s\geq q$ and $\nu \in \left[0,1\right]$, then there exists a constant $C>0$ such that
	\begin{equation*}
	\left\|u \right\|_r \leq C \left\| u \right\|^{\nu}_s \left\| u \right\|^{1-\nu}_q, \quad \forall u \in \boldsymbol{E}^s.
	\end{equation*}
	\textbf{Generalized Poincar\'e inequalities}:
	\begin{equation*}
		\lambda_1 \left\| u \right\|^2_{\kappa} \leq \left\| u \right\|^2_{\kappa+1}, \quad \forall u \in \boldsymbol{E}^{\kappa+1},
	\end{equation*}
	where $\lambda_1>0$ is the first eigenvalue of $A$.
	\begin{Definition}(\cite{hhmm1,lw})
	A two parameter family of mapping $U(t,\tau):\mathbb{X} \rightarrow \mathbb{X}$, $t \geq \tau$, $\tau \in \mathbb{R}$, is called to be a norm-to-weak continuous process if \\
	(1) $U(\tau,\tau)x=x$ for all $\tau \in \mathbb{R}$, $x \in \mathbb X$;\\
	(2) $U(t,s)U(s,\tau)x=U(t,\tau)x$ for all $t \geq s \geq \tau$, $\tau \in \mathbb{R}$, $x \in \mathbb{X}$;\\
	(3) $U(t,\tau)x_n \rightharpoonup U(t,\tau)x$ if $x_n \to x$ in $\mathbb{X}$.
	\end{Definition}
\begin{Definition}(\cite{hhmm1,lz})
	A family of bounded set $\widehat{B}=\{B(t):t \in \mathbb{R}\} \in \mathcal{D}$ is called pullback $\mathcal{D}$-absorbing for the process $\{U(t,\tau)\}$ if for any $t \in \mathbb{R}$ and for any $\widehat{D} \in \mathcal{D}$, there exists $\tau_0(t,\widehat{D}) \leq t$ such that
	\[U(t,\tau)D(\tau) \subset B(t) \quad \text{for all} \ \tau \leq \tau_0(t,\widehat{D}).\]
\end{Definition}
\begin{Definition}(\cite{hhmm1,lz})
	A family $\widehat{A}=\{A(t):t \in \mathbb{R}\} \subset \mathcal{P} (\mathbb{X})$ is said to be a pullback $\mathcal{D}$-attractor for $\{U(t,\tau)\}$ if\\
	(1) $A(t)$ is compact for all $t \in \mathbb{R}$;\\
	(2) $\widehat{A}$ is invariant; i.e., $U(t,\tau)A(\tau)=A(t)$, for all $t \geq \tau$;\\
	(3) $\widehat{A} $ is pullback $\mathcal{D}$-attracting; i.e.,
	\[\lim_{\tau \to -\infty} dist_{\mathbb{X}}(U(t,\tau)D(\tau),A(t))=0,\]
	for all $ \widehat{D} \in \mathcal{D}$ and all $t \in \mathbb{R}$;\\
	(4) If $\{C(t):t \in \mathbb{R}\}$ is another family of closed attracting sets, then $A(t) \subset C(t)$ for all $t \in \mathbb{R}$.
\begin{Remark}
		$dist_{\mathbb{X} }(A,B)$ denotes the Hausdorff semidistance between $A$ and $B$ which are bounded sets of $\mathbb X$, that is, $dist_{\mathbb{X}}(A,B)=\underset{x \in A}{sup} \ \underset{y \in B}{inf} \left\| x-y \right\|_{\mathbb X}$.
\end{Remark}
\end{Definition}
\begin{Definition}(\cite{wzl})
	A family  $\widehat{B}=\{B(t):t \in \mathbb{R}\} $ of bounded subsets is called:
	\begin{itemize}
	\item[i)] pullback absorbing, if for any $t\in \mathbb{R}$ and any bounded set $D \subset \mathbb{X}$, there exists $\tau_0(t,\widehat{B}) \leq t$ such that 
	\begin{equation*}
		U(t,\tau)D \subseteq B(t), \quad \forall \tau \leq \tau_0(t,\widehat{B});
	\end{equation*}
	\item[ii)] pullback $\mathcal{D}$-absorbing, if for any $t\in \mathbb{R}$, there exists $\tau_0(t,\widehat{B}) < t$ such that
	\begin{equation*}
		U(t,\tau)B(\tau) \subseteq B(t), \quad \forall \tau \leq \tau_0(t,\widehat{B});
	\end{equation*}
	\item[iii)] semi-uniformly pullback $\mathcal{D}$-absorbing, if for any $t\in \mathbb{R}$, there exists $\tau_0(t,\widehat{B}) < t$ such that
	\begin{equation*}
		U(s,\tau)B(\tau) \subseteq B(s), \quad \forall \tau \leq \tau_0(t,\widehat{B}), s\leq t.
	\end{equation*}
	\end{itemize}
\end{Definition}
\begin{Theorem}(\cite{hhmm1,lz})\label{Th27}
	Let  $\{U(t,\tau)\}$ be a norm-to-weak continuous process such that  $\{U(t,\tau)\}$ is pullback $\omega$-$\mathcal{D}$-limit compact. If there exists a family of pullback $\mathcal{D}$-absorbing sets $\widehat{B}=\{B(t):t \in \mathbb{R}\} \in \mathcal{D}$ for the preocess $\{U(t,\tau)\}$, then there exists a pullback $\mathcal{D}$-attractor $\{A(t):t \in \mathbb{R}\}$ such that 
	\begin{equation*}
		A(t)=\omega(\widehat{B},t)=\bigcap_{s \leq t} \overline{\bigcup_{\tau \leq s} U(t,\tau)B(\tau)}.
	\end{equation*} 
\end{Theorem}
\begin{Definition}(\cite{wq1})
	Let $\widehat{D}=\{D(t)\}_{t \in \mathbb{R}} $ be a family of bounded subsets in $\mathbb{X}$. A process $U(\cdot,\cdot)$ is said to be pullback $\mathcal{D}$-asymptotically compact, if for any $t \in \mathbb{R}$, any sequence $\tau_n \to -\infty$ and $x_n \in D(\tau_n)$, the sequence $\{U(t,\tau_n)x_n\}_{n \in \mathbb{N}}$ is precompact in $\mathbb X$.
\end{Definition}
\begin{Lemma}(\cite{wq2})
	$U(\cdot,\cdot)$ is pullback $\mathcal{D}$-asymptotic compact $\Longleftrightarrow$ $U(\cdot,\cdot)$ is pullback $\omega$-$\mathcal{D}$-limit compact.
\end{Lemma}
 Applying the ideas in \cite{wzl} to study the pullback asymptotic regularity of processes, we denote by $\mathfrak{I}$ the space of continuous increasing functions $ \xi: \mathbb{R} \to \mathbb{R}^+$, $\mathfrak{D}$ the space of continuous decreasing functions $\gamma:\mathbb{R} \to \mathbb{R}^+$ wiht $\lim\limits_{r \to \infty} \gamma(r)<1$. Let $\widehat{\mathfrak{I}}$ denote the space of functions $\alpha: \mathbb{R} \times \mathbb{R} \to \mathbb{R}^+$, which satisfies that for any $(t,r) \in \mathbb{R} \times \mathbb{R}$, $\alpha(\cdot, r) \in \mathfrak{I}$ and $\alpha( t, \cdot) \in \mathfrak{I}$. Let $\widehat{\mathfrak{D}}$ denote the space of functions $\beta: \mathbb{R} \times \mathbb{R} \to \mathbb{R}^+$, which satisfies that for any $(t,r) \in \mathbb{R} \times \mathbb{R}$, $\beta(\cdot ,r) \in \mathfrak{I}$ and $\beta(t, \cdot) \in \mathfrak{D}$. Let $\mathcal{H}$ and $\mathcal{V}$ be Banach spaces, $\{U(t,\tau) | t\geq \tau\}$ be a family of processes defined in $\mathcal{H}$. Assume $\mathcal{V}$ is continuously embedded into $\mathcal{H}$, such that closed balls in $\mathcal{V}$ are closed in $\mathcal{H}$ as well. Let the family $\widehat{B}=\{B(t)\}_{t \in \mathbb{R}}$ be semi-uniformly pullback $\mathcal{D}$-absorbing in $\mathcal{H}$ and denote by $R_0(t)=\left\| B(t) \right\|_{\mathcal{H}}$ for each $t\in \mathbb{R}$. Then we have the following lemma.
\begin{Lemma}(\cite{wzl})\label{lem28}
	Suppose for every $x\in \mathcal{H}$, there exist two operator $S_x(t,\tau)$ on $\mathcal{H}$ and $P_x(t,\tau)$ on $\mathcal{V}$ satisfy the following assumptions:
	\begin{itemize}
		\item[i)] For any $y \in \mathcal{H}$ and $ z \in \mathcal{V}$ with $y+z=x$,
		\begin{equation*}
			U(t,\tau)x=S_x(t,\tau)y+P_x(t,\tau)z, \quad \forall t \geq \tau ,
		\end{equation*}
		and $S_x(\tau,\tau)y=y$, $P_x(\tau,\tau)z=z$, $\forall \tau \in \mathbb{R}$.
		\item[ii)] There exists $\gamma \in \mathfrak{D}$ such that
		\begin{equation*}
			\sup_{x_{t-r} \in B(t-r)} \left\| S_{x_{t-r}}(t,t-r) y_{t-r} \right\|_{\mathcal{H}} \leq \gamma(r) \left\|y_{t-r} \right\|_{\mathcal{H}}, \quad \forall t\in \mathbb{R}, r \geq 0, y_{t-r} \in \mathcal{H}.
		\end{equation*}
		\item[iii)] There exist $\beta \in \widehat{\mathfrak{D}}$ and $\alpha \in \widehat{\mathfrak{I}}$ such that
		\begin{equation*}
			\sup_{x_{t-r} \in B(t-r)} \left\| P_{xt-r}(t,t-r) z_{t-r} \right\|_{\mathcal{V}} \leq \beta(t,r) \left\| z_{t-r} \right\|_{\mathcal{V}} + \alpha(t,r), \quad \forall t \in \mathbb{R}, r\geq 0. z_{t-r} \in \mathcal{V}.
		\end{equation*}
	\end{itemize}
	Then, for any $t_0 \in \mathbb{R}$, there exist positive constants $r_{\ast}$, $Q_{r_{\ast}}$ and positive function $\rho_{r_{\ast}}(t)$, such that
	\begin{equation*}
		dist_{\mathcal{H}}(U(t,t-r)B(t-r),\overline{\mathbf{B}}_{\mathcal{V}}(\rho_{r_{\ast}}(t))) \leq Q_{r_{\ast}}(\gamma(r_{\ast}))^{\frac{r}{r_{\ast}}}R_0(t-r), \quad \forall t \leq t_0, r\geq0.
	\end{equation*}
where
\begin{equation}\label{28lem01}
	\rho_{r_{\ast}}(t)=(1+\beta(t,0)-\beta(t,r_{\ast}))\frac{\alpha(t,r_{\ast})}{1-\beta(t,r_{\ast})} \quad \text{and} \quad  Q_{r_{\ast}}=\frac{\gamma(0)}{\gamma(r_{\ast})}.
\end{equation}
\end{Lemma}

\section{Existence}\label{sec3}
The existence and uniqueness of weak solution $u$ to problem \eqref{ine01} (see, e.g., \cite{hhmm1}) can be obtained by the usual Faedo-Galerkin approximation and a compactness method. Such solutions satisfy that for any $\left[\tau,T\right] \in \mathbb{R}$,
\begin{equation*}
	u \in C(\left[ \tau,T\right]; H_0^1(\Omega)) \text{,} \quad  \quad \frac{\partial u}{\partial t} \in L^2((\tau,T);H_0^1(\Omega)),
\end{equation*}
with $u(t)=\phi(t-\tau)$ for $t \in \left[ \tau-h, \tau \right]$. 
\begin{Lemma}(\cite{hhmm1}) \label{3lem01}
	For any $\tau \in \mathbb{R}, T>\tau$, $u^0 \in H_0^1(\Omega)$, $\phi \in L^2((-h,0);L^2(\Omega))$ and if there exist positive constants $\eta_1$, $\eta_2 <\frac{1}{2}$ such that $\lambda_1>\max\left\{\mu,l\right\}+\frac{\eta_1 +\eta_2}{2}+\frac{C_b}{2\eta_1}$, then problem \eqref{ine01} has a unique weak solution $u$ on $\left( \tau,T\right)$.
\end{Lemma}
Invoking Lemma \ref{3lem01}, we will apply the results in the phase space $\mathbb{Y}:=H_0^1(\Omega)\times L^2((-h,0);L^2(\Omega))$, which is a Hilbert space with the norm 
\begin{equation*}
	\left \| (u^0,\phi) \right \|_{\mathbb{Y}}^2=\left\| \nabla u^0 \right\|^2+ \int_{-h}^{0} \left\| \phi(\theta) \right\|^2 {\rm d}\theta,
\end{equation*}
with a pair $\left( u^0, \phi \right)$ of $\mathbb{Y}$. 

Thus we can construct the process $U(t,\tau)$ with respect to problem \eqref{ine01} as follows:
\begin{equation*}
	\left( u^0, \phi \right) \mapsto U(t,\tau) \left( u^0, \phi \right)=\left(u(t),u_t\right) \quad \text{for all} \ t \geq \tau \ \text{and} \ \left( u^0, \phi \right) \in \mathbb{Y},
\end{equation*}
and the mapping $U(t,\tau):\mathbb{Y} \to \mathbb{Y}$ is continuous.

Simultaneously, we need to consider the Hilbert space 
\begin{equation*}
	\mathbb{Y}_1:=H_0^1(\Omega)\times L^2((-h,0);H_0^1(\Omega)),
\end{equation*}
with the norm
\begin{equation*}
	\left \| (u^0,\phi) \right \|_{\mathbb{Y}_1}^2=\left\| \nabla u^0 \right\|^2+ \int_{-h}^{0} \left\| \nabla \phi(\theta) \right\|^2 {\rm d}\theta.
\end{equation*}

We remark that when $t-\tau \geq h$, $U(t,\tau)$ maps $\mathbb{Y}$ to $\mathbb{Y}_1$. 

Let $\mathbb{K}^{\kappa}:=\boldsymbol{E}^{\kappa} \times L^2((-h,0);L^2(\Omega))$, which is a Hilbert space with the norm 
\begin{equation*}
	\left \| (u^0,\phi) \right \|_{\mathbb{K}^{\kappa}}^2=\|  u^0 \|^2_{\kappa}+ \int_{-h}^{0} \left\| \phi(\theta) \right\|^2 {\rm d}\theta.
\end{equation*}
 It's not difficult to find that $\mathbb{K}^{1}=\mathbb{Y}$.

\begin{Lemma}\label{L32}
	Let assumptions $\left(\mathbf{H1}\right)-\left(\mathbf{H3}\right)$ be satisfied. Assume $g \in L^2_{loc}(\mathbb{R}; \boldsymbol{E}^{-1})$ and satisfies
	\begin{equation}\label{1lem08}
	G_{-1}(t):= \sup_{\tau \leq t} \left( \int_{\tau-1}^{\tau} \left\| g(s) \right\|^2_{-1} {\rm d}s\right)^{\frac{1}{2}}< \infty, \quad \forall t \in \mathbb{R}.
	\end{equation}
	Then for all $t$ for which $t \geq \tau + h$ and all $(u^0,\phi) \in \mathbb{Y}$, we have following estimates:
	\begin{equation}\label{1lem11}
		\left \| \nabla u(t)\right \|^2 \leq C\left\{ e^{-\sigma (t-\tau)} \left \| (u^0,\phi) \right \|_{\mathbb{Y}}^{6}+G_{-1}^2(t)  +1\right\},
	\end{equation}
	\begin{equation}\label{1lem12}
	 \int_{t-h}^{t} \left \| \nabla u(s)\right \|^2 {\rm d}s \leq C \bar{\nu}^{-1}e^{\sigma h}\left\{ e^{-\sigma (t-\tau)} \left \| (u^0,\phi) \right \|_{\mathbb{Y}}^{6}+G_{-1}^2(t)  +1\right\} 
	\end{equation}
	where $\bar{\nu} :=4(\bar{\delta} -\sigma-\frac{C_b}{\lambda_1})>0$, $0<\sigma<\bar{\delta}<\min \left\{ k, \frac{\lambda_1}{2\lambda_1 +1}\right\}$, $C>0$ is a positive constant which is independent of $t$ and $\tau$.
\end{Lemma}
\begin{proof}
	Multiplying \eqref{ine01} by $u+\partial_t u$ and integrating it over $\Omega$, we arrive at
		\begin{align}
			&\frac{1}{2} \frac{d}{dt}\left( \left \| u(t)\right \|^2 +2\left \| \nabla u(t)\right \|^2  \right)+\left \| \nabla u(t)\right \|^2+\left \| \partial_t u(t)\right \|^2+ \left \| \nabla \partial_t u(t)\right \|^2 \nonumber\\
			&\quad +\int_{\Omega} f(u)(u+\partial_t u){\rm d}x \nonumber\\
			&= \int_{\Omega} b(t,u_t)(u+\partial_t u){\rm d}x+\int_{\Omega} g(t)(u+\partial_t u){\rm d}x.
		\end{align}
Using \eqref{1eq05}, the Cauchy-Schwarz and Young inequalities, we observe that
\begin{equation*}
	\begin{aligned}
	&\int_{\Omega} b(t,u_t)(u+\partial_t u){\rm d}x+\int_{\Omega} g(u+\partial_t u){\rm d}x \\
	&\leq  \left \| b(t,u_t) \right \|^2 + \left \| g(t) \right \|^2_{-1} +\frac{1}{2}\left \| u(t) \right \|^2+\frac{1}{2}\left \| \partial_t u(t) \right \|^2+\frac{1}{2}\left \| u(t) \right \|^2_1+\frac{1}{2}\left \| \partial_t u(t) \right \|^2_1,
	\end{aligned}
\end{equation*}
which gives
	\begin{align}
		&\frac{1}{2} \frac{d}{dt}\left( \left \| u(t)\right \|^2 +2\left \| \nabla u(t)\right \|^2  +2\int_{\Omega} F(u(t)){\rm d}x\right)+\frac{1}{2}\left \| \nabla u(t)\right \|^2+ k\int_{\Omega} F(u(t)){\rm d}x \nonumber \\
		&\leq \left \| b(t,u_t) \right \|^2 + \left \| g(t) \right \|^2_{-1} +(\frac{1}{2}+\delta) \left \| u(t) \right \|^2+C_{\delta 1}\left| \Omega \right|.
	\end{align}
By the Poincar\'{e}'s inequality, we get
	\begin{align}
		& \frac{d}{dt}\left( \left \| u(t)\right \|^2 +2\left \| \nabla u(t)\right \|^2  +2\int_{\Omega} F(u(t)){\rm d}x\right)+\left(1- \frac{2(\frac{1}{2}+\delta)}{\lambda_1} \right)\left \| \nabla u(t)\right \|^2 \nonumber\\
		&\quad +2k\int_{\Omega} F(u(t)){\rm d}x \nonumber \\
		&\leq 2\left \| b(t,u_t) \right \|^2 + 2\left \| g(t) \right \|^2_{-1} +2C_{\delta 1}\left| \Omega \right|.
	\end{align}
We can take $\delta$ small enough such that
\begin{equation}
	\bar{\delta} \left(  \left \| u(t)\right \|^2 +2\left \| \nabla u(t)\right \|^2 \right) \leq \left(1- \frac{2(\frac{1}{2}+\delta)}{\lambda_1} \right)\left \| \nabla u(t)\right \|^2,
\end{equation}
where $0<\bar{\delta} < \min \left\{ k, \frac{\lambda_1}{2\lambda_1 +1}\right\}$. Then it follows that
\begin{equation}\label{1lem01}
	\frac{d}{dt}E(t)+\bar{\delta} E(t) \leq 2\left \| b(t,u_t) \right \|^2 + 2\left \| g(t) \right \|^2_{-1} +2C_{\delta 1}\left| \Omega \right|,
\end{equation}
where 
\begin{equation}\label{1lem02}
	E(t)=\left \| u(t)\right \|^2 +2\left \| \nabla u(t)\right \|^2  +2\int_{\Omega} F(u(t)){\rm d}x.
\end{equation}
Multiplying \eqref{1lem01} by $e^{\sigma t}$ with $0<\sigma<\bar{\delta}$, we obtain
\begin{equation*}
	\begin{aligned}
		& \frac{d}{dt}\left(e^{\sigma t}E(t)\right)\leq \left( \sigma -\bar {\delta} \right)e^{\sigma t} E(t) +2e^{\sigma t}\left( \left \| b(t,u_t) \right \|^2 + \left \| g(t) \right \|^2_{-1} +C_{\delta 1}\left| \Omega \right| \right).
	\end{aligned}
\end{equation*}
Integrating the above inequality from $\tau$ to $t$, we have
\begin{equation*}
	\begin{aligned}
		E(t)&\leq e^{-\sigma (t-\tau)} E(\tau)+ \left( \sigma -\bar {\delta} \right) e^{-\sigma t} \int_{\tau}^{t} e^{\sigma s} E(s) {\rm d}s+2e^{-\sigma t} \int_{\tau}^{t} e^{\sigma s}  \left \| b(s,u_s) \right \|^2 {\rm d}s \\
		&\quad+ 2e^{-\sigma t} \int_{\tau}^{t} e^{\sigma s}  \left \| g(s) \right \|^2_{-1} {\rm d}s  +2C_{\delta 1}\left| \Omega \right| \sigma^{-1}.
	\end{aligned}
\end{equation*}
	Therefore, using  (\uppercase \expandafter{\romannumeral 2}) and (\uppercase \expandafter{\romannumeral 5}), we obtain
	\begin{align}\label{1lem07}
		E(t)&\leq e^{-\sigma (t-\tau)} E(\tau)+ \left( \sigma -\bar {\delta} \right) e^{-\sigma t} \int_{\tau}^{t} e^{\sigma s} E(s) {\rm d}s+2C_be^{-\sigma (t-\tau)} \int_{\tau-h}^{\tau}  \left \| u(s) \right \|^2 {\rm d}s\nonumber \\
		&\quad+2C_b e^{-\sigma t} \int_{\tau}^{t} e^{\sigma s} \left \| u(s)\right \|^2 {\rm d}s+ 2e^{-\sigma t} \int_{\tau}^{t} e^{\sigma s}  \left \| g(s) \right \|^2_{-1} {\rm d}s  +2C_{\delta 1}\left| \Omega \right| \sigma^{-1}.
	\end{align}
We use \eqref{1eq06} in \eqref{1lem02} to obtain by choosing $\delta$ small enough
	\begin{align}\label{1lem05}
		E(t)&\geq \left \| u(t)\right \|^2 +2\left \| \nabla u(t)\right \|^2  -2\delta \left \| u(t)\right \|^2 -2C_{\delta 2}\left| \Omega \right| \nonumber\\
		&\geq \frac{1}{2} \left \| \nabla u(t)\right \|^2 -2C_{\delta 2}\left| \Omega \right|.
	\end{align}
	On the other hand,
\begin{equation}\label{1lem03}
	E(\tau)=\left \| u(\tau)\right \|^2 +2\left \| \nabla u(\tau)\right \|^2  +2\int_{\Omega} F(u(\tau)){\rm d}x.
\end{equation}
Thanks to \eqref{1eq07}, we get
\begin{equation*}
	\int_{\Omega} F(u){\rm d}x \leq C_2 \int_{\Omega} \left| u \right| {\rm d}x +\frac{C_2}{6} \int_{\Omega} \left| u \right|^{6}{\rm d}x+C.
\end{equation*}
Applying the Poincar\'{e}'s inequality, the H\"{o}lder's inequality, the Young inequality and the embedding $H_0^1(\Omega) \hookrightarrow L^{6}(\Omega)$, we have

	\begin{align}\label{1lem04}
		\int_{\Omega} F(u) {\rm d}x&\leq C \sqrt{\left| \Omega \right|} \left( \int_{\Omega} \left| u \right|^2{\rm d}x \right)^{\frac{1}{2}}+\frac{C}{6}\int_{\Omega} \left| u \right|^6{\rm d}x +C\nonumber\\
		& \leq C \sqrt{\left| \Omega \right|} \left\| u \right\|+\frac{C}{6}\left\| u \right\|_{L^6(\Omega)}^{6} +C\nonumber\\
		& \leq C \lambda_1^{-1} \sqrt{\left| \Omega \right|} \left\| \nabla u \right\|^{6}+\frac{C}{6}\left\| 
		\nabla u \right\|^{6} +C\nonumber\\
		&\leq C \left\| \nabla u \right\|^{6}+C
	\end{align}
where $C$ is a positive constant. Using \eqref{1lem04} in \eqref{1lem03}, the Poincar\'{e}'s inequality, and the Young inequality, we find that
\begin{equation}\label{1lem06}
	E(\tau)\leq (\lambda_1^{-1}+2)\left \| \nabla u(\tau)\right \|^2  +C \left\| \nabla u(\tau) \right\|^{6} +C\leq C \left\| \nabla u(\tau) \right\|^{6}+C.
\end{equation}
We substitute \eqref{1lem05} and \eqref{1lem06} in \eqref{1lem07} to yield
\begin{equation*}
	\begin{aligned}
		\frac{1}{2} \left \| \nabla u(t)\right \|^2 &\leq Ce^{-\sigma (t-\tau)}\left\| \nabla u(\tau) \right\|^{6}+2C_be^{-\sigma (t-\tau)} \int_{\tau-h}^{\tau}  \left \| u(s) \right \|^2 {\rm d}s+2C_b e^{-\sigma t} \int_{\tau}^{t} e^{\sigma s} \left \| u(s)\right \|^2 {\rm d}s\\
		&+\left( \sigma -\bar {\delta} \right) e^{-\sigma t} \int_{\tau}^{t} e^{\sigma s} \left( \left \| u(s)\right \|^2 +2\left \| \nabla u(s)\right \|^2  +2\int_{\Omega} F(u(s)){\rm d}x \right)  {\rm d}s\\
		&+2e^{-\sigma t} \int_{\tau}^{t} e^{\sigma s}  \left \| g(s) \right \|^2_{-1} {\rm d}s  +2C_{\delta 1}\left| \Omega \right| \sigma^{-1}+2C_{\delta 2}\left| \Omega \right|.
	\end{aligned}
\end{equation*}
Then, for $\bar{\delta} -\sigma-\frac{C_b}{\lambda_1}>0$, using \eqref{1lem08} and the Young inequality, we deduce that
	\begin{align}\label{1lem10}
		 &\left \| \nabla u(t)\right \|^2 +4(\bar{\delta} -\sigma-\frac{C_b}{\lambda_1})e^{-\sigma t} \int_{\tau}^{t} e^{\sigma s} \left \| \nabla u(s)\right \|^2 {\rm d}s \nonumber\\
		 &\leq Ce^{-\sigma (t-\tau)}\left\| \nabla u(\tau) \right\|^{6}+4C_be^{-\sigma (t-\tau)} \int_{\tau-h}^{\tau}  \left \| u(s) \right \|^2 {\rm d}s+4e^{-\sigma t} \int_{\tau}^{t} e^{\sigma s}  \left \| g(s) \right \|^2_{-1} {\rm d}s  +C.
	\end{align}
Noting that 
	\begin{align}\label{1lem09}
		e^{-\sigma t} \int_{\tau}^{t} e^{\sigma s}  \left \| g(s) \right \|^2_{-1} {\rm d}s &\leq  e^{-\sigma t} \left( e^{\sigma t} \int_{t-1}^{t} +e^{\sigma (t-1)}\int_{t-2}^{t-1}+ \cdots \right) \left \| g(s) \right \|^2_{-1} {\rm d}s \nonumber \\
		&\leq G_{-1}^2(t) \sum_{n=0}^{\infty} e^{-\sigma n} \nonumber\\
		& \leq G_{-1}^2(t) (1-e^{-\sigma})^{-1}.
	\end{align}
Hence, substituting \eqref{1lem09} into \eqref{1lem10}, we have
\begin{equation*}
	\begin{aligned}
		&\left \| \nabla u(t)\right \|^2 +4(\bar{\delta} -\sigma-\frac{C_b}{\lambda_1})e^{-\sigma t} \int_{\tau}^{t} e^{\sigma s} \left \| \nabla u(s)\right \|^2 {\rm d}s\\
		&\leq C\left\{ e^{-\sigma (t-\tau)} \left( \left\| \nabla u(\tau) \right\|^{6}+\left\| \phi \right\|_{L^2((-h,0);L^2(\Omega))}^{6} \right)+G_{-1}^2(t)  +1\right\}\\
		&\leq C\left\{ e^{-\sigma (t-\tau)} \left \| (u^0,\phi) \right \|_{\mathbb{Y}}^{6}+G_{-1}^2(t)  +1\right\} .
	\end{aligned}
\end{equation*}
Consequently, for all $t \geq \tau$, we obtain \eqref{1lem11}, and
\begin{equation*}
	\bar{\nu}e^{-\sigma t} \int_{\tau}^{t} e^{\sigma s} \left \| \nabla u(s)\right \|^2 {\rm d}s \leq C\left\{ e^{-\sigma (t-\tau)} \left \| (u^0,\phi) \right \|_{\mathbb{Y}}^{6}+G_{-1}^2(t)  +1\right\} ,
\end{equation*}
where $\bar{\nu} :=4(\bar{\delta} -\sigma-\frac{C_b}{\lambda_1})>0$. Furthermore, for $t-h \geq \tau$, we conclude as $\left[t-h,t\right] \subset \left[ \tau,t\right]$
\begin{equation*}
	\int_{\tau}^{t} e^{\sigma s} \left \| \nabla u(s)\right \|^2 {\rm d}s \geq \int_{t-h}^{t} e^{\sigma s} \left \| \nabla u(s)\right \|^2 {\rm d}s \geq e^{\sigma(t-h)} \int_{t-h}^{t} \left \| \nabla u(s)\right \|^2 {\rm d}s.
\end{equation*}
Therefore, for $t-h \geq \tau$, we obtatin \eqref{1lem12}. 
\end{proof}
Let $\mathcal{R}$ be the set of functions $\varrho : \mathbb{R} \mapsto \left(0,+\infty \right)$ such that 
\begin{equation*}
	\lim\limits_{t \to -\infty} e^{\sigma t} \varrho^{6}(t)=0.
\end{equation*}
By $\mathcal{D}$ we denote the class of all families $\widehat{D} =\left\{ D(t): t \in \mathbb{R} \right\} \subset \mathcal{P}(\mathbb{Y})$ such that $D(t) \subset \overline{\mathbf{B}}_{\mathbb{Y}} \left(0,\varrho(t)\right)$, for some $\varrho \in \mathcal{R}$. 
\begin{Lemma}\label{L33}
	Assume that assumptions $\left(\mathbf{H1}\right)-\left(\mathbf{H3}\right)$ are satisfied. Then the process $U(\cdot, \cdot)$ associated to problem \eqref{ine01} has a pullback $\mathcal{D}$-absorbing family $\widehat{B}=\{B(t):t \in \mathbb{R}\}$ in $\mathbb{Y}=H_0^1(\Omega)\times L^2((-h,0);L^2(\Omega))$, where 
	\begin{equation}
		B(t)=\left\{ \left(v^0,\varphi \right) \in \mathbb{Y}_1 : \left\| \left(v^0,\varphi \right) \right\|_{\mathbb{Y}_1} \leq R(t), \left\| \frac{d \varphi}{ds} \right\|_{L^2((-h,0);L^2(\Omega))} \leq \bar{R}(t)\right\},
	\end{equation}
	where
	\begin{equation*}
		\varrho_1(t)=C\left(G_{-1}^2(t)  +1\right),
	\end{equation*}
	\begin{equation*}
		R^2(t) =\left(1+\bar{\nu}^{-1}e^{\sigma h}\right)\varrho_1(t), \quad R(t) \geq 0,
	\end{equation*}
	\begin{equation*}
		\bar{R}^2(t)=C\varrho_1^{3}(t)+C\varrho_1 (t)+CG^2_{-1}(t), \quad \bar{R}(t) \geq 0.
	\end{equation*}
\end{Lemma}
	\begin{proof}
	Firstly, we observe that for all $t \in \mathbb{R}$,
	\begin{equation}
		B(t) \subset \left\{\left(v^0,\varphi \right) \in \mathbb{Y} :\left\| \left(v^0,\varphi \right) \right\|_{\mathbb{Y}} \leq R(t) \right\},
	\end{equation}
	with
	\begin{equation}
		\lim\limits_{t \to -\infty} e^{\sigma t} R^{6}(t)=0,
	\end{equation}
	and so $\widehat{B} \in \mathcal{D}$.
	
	Then we are now concerned with the asymptotic estimate using $R(t)$ for fixed $t \in \mathbb{R}$. It may be proved as follows. By definition, we have
	\begin{equation}
		\left\| U(t,\tau)(u^0,\phi)\right\|^2_{\mathbb{Y}_1}=\left\| \nabla u(t) \right\|^2+\int_{t-h}^{t} \left\| \nabla u(s) \right\|^2 {\rm d}s.
	\end{equation}
	From \eqref{1lem11} and \eqref{1lem12}, for any $t-h \geq \tau$ and all $(u^0,\phi) \in \mathbb{Y}$, using the definition of $\varrho_1(t)$, we obtain
	\begin{equation*}
		\begin{aligned}
			&\left\| U(t,\tau)(u^0,\phi)\right\|^2_{\mathbb{Y}_1}\\
			&\leq C\left\{ e^{-\sigma (t-\tau)} \left \| (u^0,\phi) \right \|_{\mathbb{Y}}^{6}+G_{-1}^2(t)  +1\right\}+C \bar{\nu}^{-1}e^{\sigma h}\left\{ e^{-\sigma (t-\tau)} \left \| (u^0,\phi) \right \|_{\mathbb{Y}}^{6}+G_{-1}^2(t)  +1\right\}\\
			&\leq  C e^{-\sigma(t-\tau)} \left \| (u^0,\phi) \right \|_{\mathbb{Y}}^{6} \left(1+\bar{\nu}^{-1}e^{\sigma h}\right)+(1+\bar{\nu}^{-1}e^{\sigma h}) \varrho_1(t)\\
			&\leq C e^{-\sigma(t-\tau)} \left \| (u^0,\phi) \right \|_{\mathbb{Y}}^{6} \left(1+\bar{\nu}^{-1}e^{\sigma h}\right)+R^2(t).
		\end{aligned}
	\end{equation*}
	Hence, as $e^{\sigma \tau} \rightarrow 0$ when $\tau \rightarrow -\infty$, we find that
	\begin{equation}\label{2lem04}
		\left\| U(t,\tau)(u^0,\phi)\right\|^2_{\mathbb{Y}_1} \leq R^2(t).
	\end{equation}
Next, we consider the asymptotic estimate using $\bar{R}(t)$. We assume now that $t-2h \geq \tau$. Multiplying \eqref{ine01} by $\partial_t u$ and integrating it over $\Omega$, we arrive at
\begin{equation*}
	\left\| \partial_t u\right\|^2+\frac{1}{2} \frac{d}{dt} \left( \left\| \nabla u \right\|^2+2\int_{\Omega} F(u){\rm d}x\right) +\left\|  \nabla \partial_t u\right\|^2= \langle b(t,u_t),\partial_t u \rangle+\langle g(t),\partial_t u \rangle.
\end{equation*}
By the Cauchy and Young inequalities, we observe that
\begin{equation}\label{2lem01}
	 \left\| \partial_t u\right\|^2+ \frac{d}{dt}\left(\left\| \nabla u \right\|^2+\int_{\Omega} F(u){\rm d}x\right) \leq \left\| b(t,u_t) \right\|^2 +\left\| g(t) \right\|^2_{-1}.
\end{equation}
Integrating \eqref{2lem01} over $\left[t-h,t\right]$, we notice that
\begin{equation*}
	\begin{aligned}
		&\int_{t-h}^{t} \left\| \frac{d}{dt} u(s)\right\|^2 {\rm d}s +\left( \left\| \nabla u(t)\right\|^2+ 2\int_{\Omega} F(u(t)){\rm d}x \right)\\
		&\leq \left\| \nabla u(t-h) \right\|^2+2\int_{\Omega} F(u(t-h)){\rm d}x+\int_{t-h}^{t} \left\| b(s,u_s) \right\|^2 {\rm d}s +\int_{t-h}^{t} \left\| g(s) \right\|^2_{-1} {\rm d}s.
	\end{aligned}
\end{equation*}
From (\uppercase \expandafter{\romannumeral 2}), (\uppercase \expandafter{\romannumeral 4}) and the Poincar\'{e}'s inequality, it follows
\begin{equation*}
	\begin{aligned}
		&\int_{t-h}^{t} \left\| \frac{d}{dt} u(s)\right\|^2 {\rm d}s \\
		&\leq \left\| \nabla u(t-h) \right\|^2+2\int_{\Omega} F(u(t-h)){\rm d}x+\frac{C_b}{\lambda_1}\int_{t-2h}^{t} \left\| \nabla u(s) \right\|^2 {\rm d}s +\int_{t-h}^{t} \left\| g(s) \right\|^2_{-1} {\rm d}s.
	\end{aligned}
\end{equation*}
By \eqref{1lem04}, the Young inequality, we observe that 
\begin{equation*}
	\int_{t-h}^{t} \left\| \frac{d}{dt} u(s)\right\|^2 {\rm d}s \leq C\left(\left\| \nabla u(t-h) \right\|^{6}+1+\int_{t-2h}^{t} \left\| \nabla u(s) \right\|^2 {\rm d}s+\int_{t-h}^{t} \left\| g(s) \right\|^2_{-1} {\rm d}s\right).
\end{equation*}
Now, we estimate $\left\| \nabla u(t-h) \right\|^2$. Replacing $t$ by $t-h$ in \eqref{1lem11}, we obtain 
\begin{equation*}
	\left \| \nabla u(t-h) \right \|^2 \leq C e^{-\sigma(t-h-\tau)} \left \| (u^0,\phi) \right \|_{\mathbb{Y}}^{6}+C\left(1+G^2_{-1}(t-h)) \right).
\end{equation*}
Because $t-h \leq t$, $e^{\sigma h}>1$ and $G^2_{-1}(\cdot) \in \mathfrak{I} $, we have
\begin{equation*}
	\begin{aligned}
		\left \| \nabla u(t-h) \right \|^2  \leq C e^{-\sigma(t-\tau)} \left \| (u^0,\phi) \right \|_{\mathbb{Y}}^{6}+\varrho_1(t).
	\end{aligned}
\end{equation*}
Hence, using a convexity argument, we derive that
\begin{equation}\label{2lem02}
	\begin{aligned}
	\left \| \nabla u(t-h) \right \|^{6} & \leq \left( C e^{\sigma h}e^{-\sigma(t-\tau)} \left \| (u^0,\phi) \right \|_{\mathbb{Y}}^{6}+e^{\sigma h}\varrho_1(t) \right)^{3}\\
	& \leq Ce^{-3 \sigma(t-\tau)} \left \| (u^0,\phi) \right \|_{\mathbb{Y}}^{18}+\varrho_1^{3}(t).
	\end{aligned}
\end{equation}
Taking $2h$ in place of $h$ in \eqref{1lem12}, we get that for any $\tau \leq t-2h$,
\begin{equation}\label{2lem03}
	\int_{t-2h}^{t} \left\| \nabla u(s) \right\|^2 {\rm d}s \leq Ce^{-\sigma(t-\tau)} \left \| (u^0,\phi) \right \|_{\mathbb{Y}}^{6}+C\varrho_1(t).
\end{equation}
Since \eqref{1lem08}, it is not difficult to find that
\begin{equation}\label{2lem06}
	\int_{t-h}^{t} \left\| g(s) \right\|^2_{-1} {\rm d}s \leq h G^2_{-1}(t).
\end{equation}
From \eqref{2lem02}, \eqref{2lem03} and \eqref{2lem06}, it follows that for all $\tau \leq t-2h$ and all $(u^0,\phi)\in \mathbb{Y}$,
\begin{equation*}
	\begin{aligned}
		\int_{t-h}^{t} \left\| \frac{d}{dt} u(s)\right\|^2 {\rm d}s &\leq Ce^{-3 \sigma(t-\tau)} \left \| (u^0,\phi) \right \|_{\mathbb{Y}}^{18}+Ce^{-\sigma(t-\tau)} \left \| (u^0,\phi) \right \|_{\mathbb{Y}}^{6} \\
		&\quad+C\varrho_1^{3}(t)+C\varrho_1(t)+CG^2_{-1}(t).
	\end{aligned}
\end{equation*}
Hence, for all $\tau \leq t-2h$ and for any $(u^0,\phi) \in \mathbb{Y}$, we deduce that
\begin{equation*}
	\int_{t-h}^{t} \left\| \frac{d}{dt} u(s)\right\|^2 {\rm d}s \leq Ce^{-3\sigma(t-\tau)} \left \| (u^0,\phi) \right \|_{\mathbb{Y}}^{18}+Ce^{-\sigma(t-\tau)} \left \| (u^0,\phi) \right \|_{\mathbb{Y}}^{6} +\bar{R}^2(t).
\end{equation*}
Because of $e^{\sigma \tau} \rightarrow 0$ when $\tau \rightarrow -\infty$, we obtain
\begin{equation}\label{2lem05}
	\int_{t-h}^{t} \left\| \frac{d}{dt} u(s)\right\|^2 {\rm d}s \leq \bar{R}^2(t).
\end{equation}
Consequently, combining \eqref{2lem04} with \eqref{2lem05}, it is obvious that there exists a pullback $\mathcal{D}$-absorbing family $\widehat{B}=\{B(t)\}_{t \in \mathbb{R}}$ in $\mathbb{Y}=H_0^1(\Omega)\times L^2((-h,0);L^2(\Omega))$.
\end{proof}
In order to prove the process $U(\cdot,\cdot)$ is pullback $\mathcal{D}$-asymptotically compact in $\mathbb{Y}$, applying the ideas in \cite{wzl}, we decompose the nonlinearity as $f(u)=f_0(u)+f_1(u)$, where $f_0, f_1 \in C(\mathbb{R})$ and satisfy
\begin{equation}\label{2eq01}
	f_0(u)u \geq 0, \quad \left| f_0(u)-f_0(v) \right| \leq C \left| u-v \right| \left( \left| u \right|^4+\left| v \right|^4 \right),
\end{equation}
\begin{equation}\label{2eq02}
	\left| f_0(u)-f_0(v) \right| \leq C\left( \left| u \right|^5+\left| v \right|^5 +1 \right),
\end{equation}
\begin{equation}\label{2eq03}
	\lim_{\left| u \right| \to \infty} \inf  \frac{f_1(u)}{u} > -\lambda_1, \quad \left| f_1(u) \right| \leq C(1+\left| u \right|). 
\end{equation}
Hence, for each $\tau \in \mathbb{R}$ and $(u^0,\phi) \in \mathbb{Y}$, we can decompose the solution $U(t,\tau)(u^0,\phi)=(u(t),u_t)$ in to the sum
\begin{equation}
	U(t,\tau)(u^0,\phi)= \widehat{S}(t,\tau)(u^0,\phi)+ \widehat{P}(t,\tau)(u^0,\phi),
\end{equation}
where $\widehat{S}(t,\tau)(u^0,\phi)=(\hat{v}(t),0)$ and $\widehat{P}(t,\tau)(u^0,\phi)=(\hat{w}(t),\hat{w}_t)=(\hat{w}(t),u_t)$ solve the following equations, respectively:
\begin{equation}\label{ine02}
	\begin{cases}
		\partial_t \hat{v} +A \partial_t \hat{v} + A \hat{v} +f_0(\hat{v})=0 & in \  (\tau,\infty)\times\Omega,\\
		\hat{v}(t,x)=0 & on \ (\tau,\infty)\times\partial\Omega,\\
		\hat{v} (\tau,x)=u^0(x)&\tau \in \mathbb{R},x \in \Omega,\\
		\hat{v}(\tau+\theta,x)=0 & \theta \in(-h,0),x \in \Omega,
	\end{cases}
\end{equation}
and 
\begin{equation}\label{ine03}
	\begin{cases}
		\partial_t \hat{w} +A \partial_t \hat{w} + A \hat{w} =\hat{\varPhi} & in \  (\tau,\infty)\times\Omega,\\
		\hat{w}(t,x)=0 & on \ (\tau,\infty)\times\partial\Omega,\\
		\hat{w} (\tau,x)=0&\tau \in \mathbb{R},x \in \Omega,\\
		\hat{w}(\tau+\theta,x)=\phi(\theta,x)  & \theta \in(-h,0),x \in \Omega,
	\end{cases}
\end{equation}
where 
\begin{equation*}
	\hat{\varPhi}=b(t,\hat{w}_t)+g(t)-f_0(u)+f_0(\hat{v})-f_1(u).
\end{equation*}
\begin{Lemma}\label{L34}
	Assume that assumptions $\left(\mathbf{H1}\right)-\left(\mathbf{H3}\right)$ are satisfied. There exist positive constants $\bar{C}_1$, $\delta_1$ such that
	\begin{equation}\label{4lem01}
	\| \widehat{S}(t,\tau)(u^0,\phi) \|_{\mathbb{Y}} \leq Ce^{-\frac{\delta_1}{2}(t-\tau)} \left\| \hat{v}(\tau) \right\|_1, \quad \forall t \geq \tau.
	\end{equation} 
\end{Lemma}
\begin{proof}
	Multiplying the first equation in \eqref{ine02} by $\hat{v}$ and integrating it over $\Omega$, it follows
	\begin{equation*}
	\frac{1}{2} \frac{d}{dt} \left( \left\| \hat{v} \right\|^2 + \left\| \nabla \hat{v} \right\|^2 \right)+ \left\| \nabla \hat{v} \right\|^2 +\int_{\Omega} f_0(\hat{v}) \hat{v} {\rm d}x=0 .
	\end{equation*}
Using \eqref{2eq01}, we have
\begin{equation*}
	\frac{d}{dt} \left( \left\| \hat{v} \right\|^2 + \left\| \nabla \hat{v} \right\|^2 \right)+ 2\left\| \nabla \hat{v} \right\|^2  \leq 0 .
\end{equation*}
We can choose $\delta_1$ small enough such that
\begin{equation*}
	\frac{d}{dt} \left( \left\| \hat{v} \right\|^2 + \left\| \nabla \hat{v} \right\|^2 \right)+ \delta_1 \left( \left\| \hat{v} \right\|^2 + \left\| \nabla \hat{v} \right\|^2 \right)  \leq 0.
\end{equation*}
	Then, Gronwall's inequality yields
	\begin{equation*}
		\left\| \hat{v}(t) \right\|^2 + \left\| \nabla \hat{v}(t) \right\|^2 \leq e^{-\delta_1(t-\tau)}\left(\left\| \hat{v}(\tau) \right\|^2 + \left\| \nabla \hat{v}(\tau) \right\|^2\right).
	\end{equation*}
Hence, by Poincar\'{e} inequality and the definition of $\| \widehat{S}(t,\tau)(u^0,\phi) \|_{\mathbb{Y}} $, we can get \eqref{4lem01} immediately.
\end{proof}
\begin{Lemma}\label{L35}
	Let assumptions $\left(\mathbf{H1}\right)-\left(\mathbf{H3}\right)$ be satisfied. Assume $g \in L^2_{loc}(\mathbb{R}; \boldsymbol{E}^{-1})$ and satisfies that
\begin{equation}\label{5lem11}
	G_{-2/3}(t):= \sup_{\tau \leq t} \left( \int_{\tau-1}^{\tau} \left\| g(s) \right\|^2_{-2/3} {\rm d}s\right)^{\frac{1}{2}}< \infty, \quad \forall t \in \mathbb{R}.
\end{equation}
	Then, there exist positive constants $\bar{C}_2$, $\bar{C}_3$ such that
\begin{equation}\label{5lem09}
	\| \widehat{P}(t,\tau)(u^0,\phi) \|^2_{\mathbb{K}^{4/3}} \leq \bar{C}_2\left( t-\tau +1\right) \left( G^{30}_{-2/3}(t)+1\right)(e^{\bar{C}_3\left(t-\tau +1\right)\left( G^{12}_{-1}(t)+1\right)}+1).
\end{equation}
\end{Lemma}
\begin{proof}
	It can be seen from the definition that $\| \widehat{P}(t,\tau)(u^0,\phi) \|^2_{\mathbb{K}^{4/3}}= \left\| \hat{w}(t) \right\|^2_{4/3}+\int_{t-h}^{t} \left\| \hat{w}(s) \right\|^2 {\rm d}s$, we first estimate $\int_{t-h}^{t} \left\| \hat{w}(s) \right\|^2 {\rm d}s$.
	
	Multiplying the first equation in \eqref{ine03} by $\hat{w}$ and integrating it over $\Omega$, we have
	\begin{equation}\label{5lem01}
		\frac{1}{2} \frac{d}{dt}\left( \| \hat{w} \|^2 +\| \nabla \hat{w} \|^2 \right)+ \| \nabla \hat{w} \|^2 \leq \langle \hat{\varPhi}(t), \hat{w} \rangle.
	\end{equation}
	In light of \eqref{2eq01}-\eqref{2eq03} and the following Sobolev embedding
	\begin{equation*}
		\boldsymbol{E}^1 \hookrightarrow L^6(\Omega),
	\end{equation*}
	we deduce that 
		\begin{align}\label{5lem02}
			&| \langle \hat{\varPhi}(t), \hat{w} \rangle| \nonumber \\
			&\leq |\langle b(t,\hat{w}_t), \hat{w} \rangle |+| \langle g(t), \hat{w} \rangle |+| \langle f_0(u)-f_0(\hat{v}), \hat{w} \rangle |+| \langle f_1(u), \hat{w} \rangle| \nonumber \\
			&\leq \left\| g(t) \right\|_{-1}\left\| \hat{w} \right\|_1+\left\| b(t,\hat{w}_t) \right\|\left\| \hat{w} \right\|  +C\int_{\Omega}(\left| u \right|^5+\left| \hat{v} \right|^5 +1) \left| \hat{w}\right| {\rm d}x \nonumber\\
			&\quad + C\int_{\Omega}(1+\left| u \right| ) \left| \hat{w}\right| {\rm d}x \nonumber \\
			& \leq \left\| g(t) \right\|^2_{-1}+\frac{1}{4} \left\| \hat{w} \right\|^2_1+\left\| b(t,\hat{w}_t) \right\|^2 +\frac{1}{4} \left\| \hat{w} \right\|^2+C\left( \left\| u \right\|^5_{L^6}+\left\| \hat{v} \right\|^5_{L^6}+1\right)^2 \nonumber \\
			& \quad +\frac{1}{4} \left\| \hat{w} \right\|^2_{L^6}+C(1+\left\| u \right\|^2)+\frac{1}{4} \left\| \hat{w} \right\|^2 \nonumber\\
			&\leq \left\| g(t) \right\|^2_{-1}+\left\| b(t,\hat{w}_t) \right\|^2+C\left( \left\| u \right\|^{10}_{1}+\left\| \hat{v} \right\|^{10}_{1}+1\right)+\left( \frac{1}{2}+\frac{1}{2\lambda_1}\right)\left\| \hat{w} \right\|^2_1.
		\end{align}
	By \eqref{1lem11},\eqref{4lem01} and recall that $R( \cdot) \in \mathfrak{I}$, we get
		\begin{align}\label{5lem03}
		&\left\| u \right\|^{10}_{1}+\left\| \hat{v} \right\|^{10}_{1}+1 \nonumber\\
		&\leq C e^{-5\sigma(t-\tau)}R^{30}(\tau)+Ce^{-5\delta_1(t-\tau)}R^{10}(\tau)+C \left( G^{10}_{-1}(t)+1\right) \nonumber\\
		&\leq C \left(e^{-5\sigma(t-\tau)}+e^{-5\delta_1(t-\tau)}+1 \right) \left( G^{30}_{-1}(t)+1\right).
		\end{align}
	From \eqref{5lem01}, \eqref{5lem02} and \eqref{5lem03}, we arrive at 
	\begin{equation*}
		\begin{aligned}
		& \frac{d}{dt}\left( \| \hat{w} \|^2 +\| \nabla \hat{w} \|^2 \right)+ \left(1-\frac{1}{\lambda_1}\right)\| \nabla \hat{w} \|^2 \\
		&\leq 2\left\| g(t) \right\|^2_{-1}+2\left\| b(t,\hat{w}_t) \right\|^2+C \left(e^{-5\sigma(t-\tau)}+e^{-5\delta_1(t-\tau)}+1 \right) \left( G^{30}_{-1}(t)+1\right).
		\end{aligned}
	\end{equation*}
	Integrating this last estimate over $[\tau,t]$, we find that
	\begin{equation*}
		\begin{aligned}
			&\| \hat{w}(t) \|^2 +\| \nabla \hat{w}(t) \|^2+\left(1-\frac{1}{\lambda_1}\right) \int_{\tau}^{t}\| \nabla \hat{w}(s) \|^2 {\rm d}s \\
			&\leq \| \hat{w}(\tau) \|^2 +\| \nabla \hat{w}(\tau) \|^2+2 \int_{\tau}^{t} \left\| g(s) \right\|^2_{-1} {\rm d}s + 2  \int_{\tau}^{t} \left\| b(s,\hat{w}_s) \right\|^2 {\rm d}s \\
			&+  C\int_{\tau}^{t} \left(e^{-5\sigma(s-\tau)}+e^{-5\delta_1(s-\tau)} +1\right) \left( G^{30}_{-1}(s)+1\right) {\rm d}s.
		\end{aligned}
	\end{equation*}
Then, due to $\hat{w}(\tau)=0$, (\uppercase \expandafter{\romannumeral 2}) and  (\uppercase \expandafter{\romannumeral 4}), it follows that
	\begin{equation*}
	\begin{aligned}
		&\left(1-\frac{1}{\lambda_1}\right) \int_{\tau}^{t}\| \nabla \hat{w}(s) \|^2 {\rm d}s \\
		&\leq 2 \int_{\tau}^{t} \left\| g(s) \right\|^2_{-1} {\rm d}s + 2C_b  \int_{\tau}^{t} \left\| \hat{w}(s) \right\|^2 {\rm d}s+2C_b  \int_{\tau-h}^{\tau} \left\| \hat{w}(s) \right\|^2 {\rm d}s \\
		&+  C\int_{\tau}^{t} \left(e^{-5\sigma(s-\tau)}+e^{-5\delta_1(s-\tau)}+1 \right) \left( G^{30}_{-1}(s)+1\right) {\rm d}s.
	\end{aligned}
	\end{equation*}
Then 
	\begin{equation*}
	\begin{aligned}
		&\left(1-\frac{1+2C_b}{\lambda_1}\right) \int_{\tau}^{t}\| \nabla \hat{w}(s) \|^2 {\rm d}s \\
		&\leq C(t-\tau)G^2_{-1}(t) + C \left\| \phi \right\|^2_{L^2((h,0),L^2(\Omega))}  +C\left( t-\tau +1\right) \left( G^{30}_{-1}(t)+1\right)\\
		&\leq C(t-\tau)G^2_{-1}(t) +R^2(\tau)+C\left( t-\tau +1\right) \left( G^{30}_{-1}(t)+1\right).
	\end{aligned}
\end{equation*}
Hence
\begin{equation*}
	\int_{\tau}^{t}\| \nabla \hat{w}(s) \|^2 {\rm d}s \leq C\left( t-\tau +1\right) \left( G^{30}_{-1}(t)+1\right),
\end{equation*}
and when $t-h \geq \tau$
\begin{equation}\label{5lem07}
	\int_{t-h}^{t}\| \nabla \hat{w}(s) \|^2 {\rm d}s \leq C\left( t-\tau +1\right) \left( G^{30}_{-1}(t)+1\right).
\end{equation}
Next, we will estimate $\left\| \hat{w}(t) \right\|^2_{4/3}$. Multiplying the first equation in \eqref{ine03} by $A^{1/3}\hat{w}$ and integrating it over $\Omega$, we have
\begin{equation}\label{5lem04}
	\frac{1}{2} \frac{d}{dt}\left( \| A^{1/6}\hat{w} \|^2 +\| A^{2/3} \hat{w} \|^2 \right)+ \| A^{2/3}\hat{w} \|^2 \leq \langle \hat{\varPhi}(t), A^{1/3}\hat{w} \rangle.
\end{equation}
In light of \eqref{2eq01}-\eqref{2eq03} and the following Sobolev embeddings
\begin{equation}\label{5lem10}
	\boldsymbol{E}^{2/3} \hookrightarrow L^{18/5}(\Omega), \ \boldsymbol{E}^1 \hookrightarrow L^6(\Omega), \ \boldsymbol{E}^{4/3} \hookrightarrow L^{18}(\Omega), 
\end{equation}
we deduce that 
	\begin{align}\label{5lem05}
		&| \langle \hat{\varPhi}(t), A^{1/3}\hat{w} \rangle| \nonumber\\
		&\leq |\langle b(t,\hat{w}_t), A^{1/3}\hat{w} \rangle |+| \langle g(t), A^{1/3}\hat{w} \rangle |+| \langle f_0(u)-f_0(\hat{v}), A^{1/3}\hat{w} \rangle |+| \langle f_1(u),A^{1/3}\hat{w} \rangle| \nonumber\\
		&\leq \left\| b(t,\hat{w}_t) \right\| \| A^{1/3}\hat{w} \|+\left\| g(t) \right\|_{-2/3}\left\| \hat{w} \right\|_{4/3}C\int_{\Omega}\left| \hat{w}\right|(\left| u \right|^4+\left| \hat{v} \right|^4) | A^{1/3}\hat{w}| {\rm d}x \nonumber\\
		&\quad + C\int_{\Omega}(1+\left| u \right| ) | A^{1/3}\hat{w}| {\rm d}x \nonumber \\
		& \leq C\left\| g(t) \right\|^2_{-2/3}+C \left\| \hat{w} \right\|^2_{4/3}+\left\| b(t,\hat{w}_t) \right\|^2 +C \| A^{1/3}\hat{w} \|^2+C \left( 1+ \|u\| \right) \| A^{1/3}\hat{w}\| \nonumber\\
		&\quad +C\| \hat{w} \|_{L^{18}} \left( \| u\|^4_{L^6}+\| \hat{v}\|^4_{L^6}\right)\|A^{1/3} \hat{w} \|_{L^{18/5}} \nonumber\\
		&\leq C\left\| g(t) \right\|^2_{-2/3}+\left\| b(t,\hat{w}_t) \right\|^2+C\left( \| u\|^4_{1}+\| \hat{v}\|^4_{1}+1\right)\left\| \hat{w} \right\|^2_{4/3}+C.
	\end{align}
By \eqref{1lem11}, \eqref{4lem01} and recall that $R( \cdot) \in \mathfrak{I}$, we get
	\begin{align}\label{5lem06}
		&\left\| u \right\|^{4}_{1}+\left\| \hat{v} \right\|^{4}_{1}+1 \nonumber\\
		&\leq C e^{-2\sigma(t-\tau)}R^{12}(\tau)+Ce^{-2\delta_1(t-\tau)}R^{4}(\tau)+C \left( G^{4}_{-1}(t)+1\right) \nonumber\\
		&\leq C \left(e^{-2\sigma(t-\tau)}+e^{-2\delta_1(t-\tau)}+1 \right) \left( G^{12}_{-1}(t)+1\right).
	\end{align}
From\eqref{5lem04}, \eqref{5lem05} and \eqref{5lem06}, we have
\begin{equation*}
	\begin{aligned}
		&\frac{d}{dt}\left( \| \hat{w} \|^2_{1/3} +\|  \hat{w} \|^2_{4/3} \right)\\
		&\leq C \left(e^{-2\sigma(t-\tau)}+e^{-2\delta_1(t-\tau)}+1 \right) \left( G^{12}_{-1}(t)+1\right)\left( \| \hat{w} \|^2_{1/3} +\|  \hat{w} \|^2_{4/3} \right)\\
		&\quad +C\left\| g(t) \right\|^2_{-2/3}+\left\| b(t,\hat{w}_t) \right\|^2+C.
	\end{aligned}
\end{equation*}
Applying Gronwall's inequality and noting that  (\uppercase \expandafter{\romannumeral 2}), (\uppercase \expandafter{\romannumeral 4}) and \eqref{5lem07}, we arrive at
	\begin{align}\label{5lem08}
		 &\| \hat{w} \|^2_{1/3} +\|  \hat{w} \|^2_{4/3} \nonumber\\
		 &\leq C \left( \int_{\tau}^{t} \left(\left\| g(s) \right\|^2_{-2/3}+\left\| b(s,\hat{w}_s) \right\|^2+1\right){\rm d}s \right) e^{C \int_{\tau}^{t} \left(e^{-2\sigma(s-\tau)}+e^{-2\delta_1(s-\tau)}+1 \right) \left( G^{12}_{-1}(s)+1\right){\rm d}s}\nonumber\\
		 & \leq C \left( (t-\tau)G^{2}_{-2/3}(t)+(t-\tau)+\int_{\tau}^{t} \left\| b(s,\hat{w}_s) \right\|^2{\rm d}s \right)e^{C\left(t-\tau +1\right)\left( G^{12}_{-1}(t)+1\right)}\nonumber\\
		 &\leq C \left( (t-\tau)(G^{2}_{-2/3}(t)+1)+C\left( t-\tau +1\right) \left( G^{30}_{-1}(t)+1\right) +R^2(\tau) \right) e^{C\left(t-\tau +1\right)\left( G^{12}_{-1}(t)+1\right)}\nonumber\\
		 & \leq C\left( t-\tau +1\right) \left( G^{30}_{-2/3}(t)+1\right)e^{C\left(t-\tau +1\right)\left( G^{12}_{-1}(t)+1\right)}.
	\end{align}
Combining \eqref{5lem07} and \eqref{5lem08}, we conclude that
\begin{equation*}
	\begin{aligned}
		&\| \widehat{P}(t,\tau)(u^0,\phi) \|^2_{\mathbb{K}^{4/3}} \\
		&\leq C\left( t-\tau +1\right) \left( G^{30}_{-2/3}(t)+1\right)e^{C\left(t-\tau +1\right)\left( G^{12}_{-1}(t)+1\right)}+C\left( t-\tau +1\right) \left( G^{30}_{-1}(t)+1\right) \\
		&\leq C\left( t-\tau +1\right) \left( G^{30}_{-2/3}(t)+1\right)(e^{C\left(t-\tau +1\right)\left( G^{12}_{-1}(t)+1\right)}+1),
	\end{aligned}
\end{equation*}
which implies \eqref{5lem09}.
\end{proof}
Lemmas \ref{L34}, \ref{L35} imply that $\{U(t,\tau)\}$ is pullback $\mathcal{D}$-asymptotically compact in $\mathbb{K}^1$. Therefore, we conclude the followig result by Theorem \ref{Th27}.
\begin{Theorem}\label{Th36}
	The process $\{U(t,\tau)\}$ has a pullback $\mathcal{D}$-attractor $\mathcal{A}=\{A(t):t \in \mathbb{R}\} \in \mathbb{K}^1 $, and satisfies
	\begin{equation*}
		A(t)=\bigcap_{s \leq t} \overline{\bigcup_{\tau \leq s} U(t,\tau)B(\tau)}^{\mathbb{K}^1}.
	\end{equation*}
\end{Theorem}
\section{Regularity}\label{sec4}
To establish the regularity of the pullback $\mathcal{D}$-attractor $\mathcal{A}=\{A(t):t \in \mathbb{R}\}$, we will use Lemma \ref{lem28}.

\begin{Lemma}\label{L41}
	Undeer assumptions $\left(\mathbf{H1}\right)-\left(\mathbf{H3}\right)$, assume $g(t) \in L^2_{loc}(\mathbb{R};\boldsymbol{E}^{-1})$ and \eqref{5lem11} holds true. Then $\bigcup_{s \leq t}A(s)$ is bounded in $\mathbb{K}^{4/3}=\boldsymbol{E}^{4/3} \times L^2((-h,0),L^2(\Omega))$ for every $t \in \mathbb{R}$. 
\end{Lemma}
\begin{proof}
	Let $\tau \in \mathbb{R}$ and $x_{\tau}=(u^0,\phi) \in B(\tau)$. For any $y_{\tau} \in B(\tau)$ and $z_{\tau} \in \mathbb{K}^{4/3}$ satisfying $x_{\tau}=y_{\tau}+z_{\tau}$, we decompose the solution $U(t,\tau)x_{\tau}=(u(t),u_t)$ into the sum
	\begin{equation*}
		U(t,\tau)x_{\tau}=S_{x_{\tau}}(t,\tau)y_{\tau}+P_{x_{\tau}}(t,\tau)z_{\tau},
	\end{equation*}
where $S_{x_{\tau}}(t,\tau)y_{\tau}=(v(t),0)$ and $P_{x_{\tau}}(t,\tau)z_{\tau}=(w(t),w_t)=(w(t),u_t)$ solve the following problems, respectively
\begin{equation}\label{ine04}
	\begin{cases}
		\partial_t v +A \partial_t v + A v +f_0(v)=0 & in \  (\tau,\infty)\times\Omega,\\
		v(t,x)=0 & on \ (\tau,\infty)\times\partial\Omega,\\
		v (\tau,x)=v_{\tau}&\tau \in \mathbb{R},x \in \Omega,\\
		v(\tau+\theta,x)=0 & \theta \in(-h,0),x \in \Omega,
	\end{cases}
\end{equation}
and 
\begin{equation}\label{ine05}
	\begin{cases}
		\partial_t w +A \partial_t w + A w =\varPhi & in \  (\tau,\infty)\times\Omega,\\
		w(t,x)=0 & on \ (\tau,\infty)\times\partial\Omega,\\
		w (\tau,x)=w_{\tau}&\tau \in \mathbb{R},x \in \Omega,\\
		w(\tau+\theta,x)=\phi(\theta,x)  & \theta \in(-h,0),x \in \Omega,
	\end{cases}
\end{equation}
where
\begin{equation*}
	\varPhi=b(t,w_t)+g(t)-f_0(u)+f_0(v)-f_1(u) \quad \quad \text{and} \quad \quad v_{\tau}+w_{\tau}=u(\tau)=u^0.
\end{equation*}
From the above construction, we see that assumption $i)$ of Lemma \ref{lem28} holds.

Multiplying the first equation in \eqref{ine04} by $v$ and integrating it over $\Omega$, similar completely to \eqref{4lem01}, we have 
\begin{equation}\label{6lem10}
	\left\| S_{x_{\tau}}(t,\tau)y_{\tau} \right\|_{\mathbb{K}^1} \leq \gamma(t-\tau) \| y_{\tau} \|_{\mathbb{K}^1},
\end{equation}
where 
\begin{equation}
\gamma(t-\tau)=Ce^{-\frac{\delta_1}{2}(t-\tau)}.
\end{equation}
This implies assumption $ii)$ of Lemma \ref{lem28} holds.

Next, we will verify assumption $iii)$ of Lemma \ref{lem28}. By the definition, we know that $\| P_{x_{\tau}}(t,\tau)z_{\tau} \|^2_{\mathbb{K}^{4/3}} \\= \left\| w(t) \right\|^2_{4/3}+\int_{t-h}^{t} \left\| w(s) \right\|^2 {\rm d}s$. Above all, we estimate $\int_{t-h}^{t} \left\| w(s) \right\|^2 {\rm d}s$. Multiplying the first equation in \eqref{ine05} by $w$ and integrating it over $\Omega$, it follows that
\begin{equation}
	\frac{1}{2} \frac{d}{dt} \left( \left\|w \right\|^2+ \left\| \nabla w \right\|^2 \right)+ \left\| \nabla w \right\|^2 \leq \langle \varPhi(t), w \rangle.
\end{equation}
Same as \eqref{5lem02} and \eqref{5lem03}, we have 
\begin{equation}
	| \langle \varPhi(t), w \rangle | \leq \left\| g(t) \right\|^2_{-1}+ \left\| b(t,w_t) \right\|^2+ C \left( \|u \|^{10}_1+ \|v\|^{10}_1+1\right)+\left( \frac{1}{2} +\frac{1}{2\lambda_1} \right) \| w\|^2_1,
\end{equation}
and 
\begin{equation}
	\|u\|^{10}_1+\|v\|^{10}_1+1 \leq C \left( e^{-5 \sigma(t-\tau)}+e^{-5 \delta_1 (t-\tau)}+1\right)\left(G^{30}_{-1}(t)+1\right).
\end{equation}
We arrive at
	\begin{equation*}
	\begin{aligned}
		& \frac{d}{dt}\left( \| w \|^2 +\| \nabla w \|^2 \right)+ \left(1-\frac{1}{\lambda_1}\right)\| \nabla w \|^2 \\
		&\leq 2\left\| g(t) \right\|^2_{-1}+2\left\| b(t,w_t) \right\|^2+C \left(e^{-5\sigma(t-\tau)}+e^{-5\delta_1(t-\tau)}+1 \right) \left( G^{30}_{-1}(t)+1\right).
	\end{aligned}
\end{equation*}
	Integrating above estimate over $[\tau,t]$, we find that
\begin{equation*}
	\begin{aligned}
		&\| w(t) \|^2 +\| \nabla w(t) \|^2+\left(1-\frac{1}{\lambda_1}\right) \int_{\tau}^{t}\| \nabla w(s) \|^2 {\rm d}s \\
		&\leq \| w(\tau) \|^2 +\| \nabla w(\tau) \|^2+2 \int_{\tau}^{t} \left\| g(s) \right\|^2_{-1} {\rm d}s + 2  \int_{\tau}^{t} \left\| b(s,w_s) \right\|^2 {\rm d}s \\
		&+  C\int_{\tau}^{t} \left(e^{-5\sigma(s-\tau)}+e^{-5\delta_1(s-\tau)} +1\right) \left( G^{30}_{-1}(s)+1\right) {\rm d}s.
	\end{aligned}
\end{equation*}
Similar  to \eqref{5lem07} and noting that the continuous embeddings $\boldsymbol{E}^s \hookrightarrow \boldsymbol{E}^r$ for any $s>r$, when $t-h \geq \tau$, we have
\begin{equation}\label{6lem01}
	\int_{t-h}^{t}\| \nabla w(s) \|^2 {\rm d}s \leq C\left( t-\tau +1\right) \left( G^{30}_{-1}(t)+1\right)+C\| w(\tau) \|^2_{4/3}.
\end{equation}
Next, we will estimate $\left\| w(t) \right\|^2_{4/3}$. Multiplying the first equation in \eqref{ine03} by $A^{1/3}w$ and integrating it over $\Omega$, we observe that
\begin{equation}
	\frac{1}{2} \frac{d}{dt}\left( \| A^{1/6} w \|^2+ \| A^{2/3} w \|^2 \right)+\| A^{2/3}w\|^2 \leq \langle \varPhi(t), A^{1/3} w \rangle.
\end{equation}
In light of \eqref{2eq01}-\eqref{2eq03} and \eqref{5lem10}, we deduce that
	\begin{align}
		& | \langle \varPhi(t),A^{1/3}w \rangle| \nonumber \\
		&\leq | \langle g(t), A^{1/3}w \rangle|+| \langle b(t,w_t),A^{1/3}w \rangle |+ | \langle f_0(u)-f_0(v), A^{1/3}w \rangle| +| \langle f_1(u), A^{1/3}w \rangle  \nonumber\\
		&\leq \|g(t) \|_{-2/3} \|w \|_{4/3}+ \| b(t,w_t) \|^2+\frac{1}{4} \| A^{1/3} w\|^2+C(1+\|u \|)\|A^{1/3}w\| \nonumber \\
		&\quad +C\int_{\Omega} |w| (|u|^4+|v|^4)|A^{1/3}w| {\rm d}x \nonumber\\
		& \leq \|g(t) \|_{-2/3} \|w \|_{4/3}+ \| b(t,w_t) \|^2+\frac{1}{4} \| A^{1/3} w\|^2+C(1+\|u \|)\|A^{1/3}w\| \nonumber \\
		&\quad +C\int_{\Omega} |w| (|\hat{v}|^4+|v|^4)|A^{1/3}w| {\rm d}x +C\int_{\Omega} (|u|+|v|) | \hat{w}|^4 |A^{1/3}w| {\rm d}x\nonumber \\
		& \leq \|g(t) \|_{-2/3} \|w \|_{4/3}+ \| b(t,w_t) \|^2+\frac{1}{4} \| A^{1/3} w\|^2+C(1+\|u \|)\|A^{1/3}w\| \nonumber \\
		&\quad +C\|w\|_{L^{18}}(\|\hat{v}\|^4_{L^6}+\|v\|^4_{L^6})\| A^{1/3} w\|_{L^{18/5}}+C( \|u\|_{L^2}+\|v\|_{L^2}) \|\hat{w} \|^4_{L^{18}} \|A^{1/3}w\|_{L^{18/5}} \nonumber \\
		&\leq \|g(t) \|_{-2/3} \|w \|_{4/3}+ \| b(t,w_t) \|^2+\frac{1}{4} \| w\|^2_{4/3}+C(1+\|u \|_1)\|w\|_{4/3}  \nonumber \\
		&\quad +C( \|\hat{v}\|^4_1+\|v\|^4_1)\|w\|^2_{4/3}+C(\|u\|_1+\|v\|_1)\|\hat{w}\|^4_{4/3} \|w\|_{4/3} \nonumber  \\
		&\leq \left( \frac{1}{2}+C(\|\hat{v}\|^4_1+\|v\|^4_1)\right)\| w \|^2_{4/3}+C((\|u\|^2_1+\|v\|^2_1)\|\hat{w}\|^8_{4/3}+\|u\|^2_1) \nonumber   \\
		&\quad +C(\|g(t) \|^2_{-2/3}+\| b(t,w_t) \|^2+1).
	\end{align}
Noting that $G_{-1}(t) \leq CG_{-2/3}(t)$, by Lemma \ref{L32}-\ref{L35}, we have
\begin{equation}
	C(\|\hat{v}\|^4_1+\|v\|^4_1) \leq h(t,t-\tau),
\end{equation}
\begin{equation}
	C((\|u\|^2_1+\|v\|^2_1)\|\hat{w}\|^8_{4/3}+\|u\|^2_1) \leq Q(t,t-\tau),
\end{equation}
where
\begin{equation*}
	h(t,t-\tau)=Ce^{-2 \delta_1(t-\tau)}(G^4_{-2/3}(t)+1),
\end{equation*}
\begin{equation*}
	Q(t,t-\tau)=C(t-\tau +1)^4(G^{126}_{-2/3}(t)+1)(e^{C(t-\tau+1)(G^{12}_{-2/3}(t)+1)}+1),
\end{equation*}
for some constant $C>0$ large enough.

Obviously, there exits a constant $\delta_2>0$ small enough such taht 
\begin{equation*}
	\begin{aligned}
	\frac{d}{dt} ( \|w(t)\|^2_{1/3}+\| w(t)\|^2_{4/3}) &\leq (h(t,t-\tau)-\delta_2)( \|w(t)\|^2_{1/3}+\| w(t)\|^2_{4/3})\\
	&\quad +C(Q(t,t-\tau)+\|g(t) \|^2_{-2/3}+\| b(t,w_t) \|^2+1).
	\end{aligned}
\end{equation*}
Then, from Gronwall's inequality and $Q(\cdot,\cdot) \in \widehat{\mathfrak{I}}$, it derive that
	\begin{align}\label{6lem02}
		\|w(t)\|^2_{1/3}+\| w(t)\|^2_{4/3} & \leq \left(( \|w(\tau)\|^2_{1/3}+\| w(\tau)\|^2_{4/3})+\int_{\tau}^{t}C(Q(s,s-\tau)+1)  {\rm d}s\right. \nonumber \\
		&\quad \left.+C\int_{\tau}^{t} \|g(s) \|^2_{-2/3} {\rm d}s+C\int_{\tau}^{t}\| b(s,w_s) \|^2 {\rm d}s\right)e^{\int_{\tau}^{t} (h(s,s-\tau)-\delta_2){\rm d}s}.
	\end{align}
Due to \eqref{6lem01},(\uppercase \expandafter{\romannumeral 2}) and (\uppercase \expandafter{\romannumeral 4}), we observe that
\begin{equation}\label{6lem03}
	e^{\int_{\tau}^{t} h(s,s-\tau){\rm d}s}	\leq e^{(G^4_{-2/3}(t)+1) \int_{\tau}^{t}Ce^{-2 \delta_1(t-\tau)} {\rm d}s} \leq e^{C(G^4_{-2/3}(t)+1)},
\end{equation}
\begin{equation}\label{6lem04}
	e^{-\delta_2(t-\tau)} \int_{\tau}^{t} \|g(t) \|^2_{-2/3} {\rm d}s\leq \int_{\tau}^{t}e^{-\delta_2(t-s)} \|g(t) \|^2_{-2/3} {\rm d}s \leq C G^2_{-2/3}(t),
\end{equation}
\begin{equation}\label{6lem05}
	e^{-\delta_2(t-\tau)} \int_{t-h}^{t}\| \nabla w(s) \|^2 {\rm d}s \leq C\left( t-\tau +1\right)e^{-\delta_2(t-\tau)} \left( G^{30}_{-1}(t)+1\right)+Ce^{-\delta_2(t-\tau)}\| w(\tau) \|^2_{4/3}.
\end{equation}
We now can find that
\begin{equation*}
	C\left( t-\tau +1\right)e^{-\delta_2(t-\tau)} \left( G^{30}_{-1}(t)+1\right) \leq Q(t,t-\tau).
\end{equation*}
Substituting \eqref{6lem03}-\eqref{6lem05} into \eqref{6lem02}, we derive that 
	\begin{align}\label{6lem06}
		\| w(t)\|^2_{4/3} &\leq Ce^{-\delta_2 (t-\tau) }e^{C(G^4_{-2/3}(t)+1)}\| w(\tau)\|^2_{4/3} \nonumber \\
		& \quad +C(1+Q(t,t-\tau)+G^2_{-2/3}(t))e^{C(G^4_{-2/3}(t)+1)}.
	\end{align}
 Combining \eqref{6lem01} and \eqref{6lem06} and using $\|z_{\tau} \|^2_{4/3}=\| w(\tau)\|^2_{4/3}+\left\| \phi \right\|_{L^2((-h,0);L^2(\Omega))}^{2}$, we have
\begin{equation*}
	\begin{aligned}
		&\| P_{x_{\tau}}(t,\tau)z_{\tau} \|^2_{4/3}\\
		&\leq Ce^{-\delta_2 (t-\tau) }e^{C(G^4_{-2/3}(t)+1)}\| w(\tau)\|^2_{4/3}+C\| w(\tau) \|^2_{4/3}\\
		& \quad +C(1+Q(t,t-\tau)+G^2_{-2/3}(t))e^{C(G^4_{-2/3}(t)+1)}+C\left( t-\tau +1\right) \left( G^{30}_{-1}(t)+1\right).
	\end{aligned}
\end{equation*}
Then, we can conclude that
\begin{equation}\label{6lem07}
	\| P_{x_{\tau}}(t,\tau)z_{\tau} \|_{4/3} \leq \beta(t,t-\tau) \|z_{\tau} \|_{4/3}+\alpha(t,t-\tau),
\end{equation}
where 
\begin{equation*}
	\beta(t,t-\tau)=Ce^{-\frac{\delta_2(t-\tau)}{2}}e^{C(G^4_{-2/3}(t)+1)},
\end{equation*}
\begin{equation*}
	\alpha(t,t-\tau)=C(1+Q^{1/2}(t,t-\tau)+G_{-2/3}(t))e^{C(G^4_{-2/3}(t)+1)},
\end{equation*}
for some $C>0$ large enough.

Since $G_{-2/3}(\cdot) \in \mathfrak{I}$ and $Q(\cdot,\cdot) \in \widehat{\mathfrak{I}}$, we know that $\beta(\cdot,\cdot) \in \widehat{\mathfrak{D}}$ and $\alpha(\cdot,\cdot) \in \widehat{\mathfrak{I}}$. Hence, taking $\tau=t-r(r \geq 0)$ and $z_{t-r} \in \boldsymbol{E}^{4/3}$ in \eqref{6lem07}, it is easy to check that assumption $iii)$ of Lemma \ref{lem28} is satisfied.

Let $T_0 \in \mathbb{R}$ be arbitrary fixed. Recall that the family $\widehat{B}=\{B(t)\}_{t \in \mathbb{R}}$ is pullback $\mathcal{D}$-abosorbing and $G_{-2/3}(\cdot) \in \mathfrak{I}$, we can choose $r_{\ast}:=r_{\ast}(T_0)>0$ large enough such that 
\begin{equation}
	U(t,t-r)B(t-r) \subseteq B(t),  \quad \forall r \geq r_{\ast}, \ t\leq T_0, 
\end{equation}
and the following estimates hold ture
\begin{equation}
	\gamma(r_{\ast})=Ce^{-\frac{\delta_1 r_{\ast}}{2}}<1,
\end{equation}
\begin{equation}
	\beta(t,r_{\ast})=Ce^{-\frac{\delta_2 r_{\ast}}{2}}e^{C(G^4_{-2/3}(t)+1)} \leq \beta(T_0,r_{\ast})=Ce^{-\frac{\delta_2 r_{\ast}}{2}}e^{C(G^4_{-2/3}(T_0)+1)} <1, \quad  \forall t \leq T_0,
\end{equation}
\begin{equation}\label{6lem11}
	\frac{1}{1-Ce^{-\frac{\delta_2 r_{\ast}}{2}}e^{C(G^4_{-2/3}(t)+1)} }\leq 2, \quad  \forall t \leq T_0.
\end{equation}
From \eqref{28lem01}, setting
\begin{equation*}
	\begin{aligned}
		\rho_{r_{\ast}}(t) & =  \frac{1+\beta(t,0)-\beta(t,r_{\ast})}{1-\beta(t,r_{\ast})} \alpha(t,r_{\ast})\\
		&=\frac{C\left( 1+e^{C(G^4_{-2/3}(t)+1)}-e^{-\frac{\delta_2 r_{\ast}}{2}}e^{C(G^4_{-2/3}(t)+1)}\right)}{1-Ce^{-\frac{\delta_2 r_{\ast}}{2}}e^{C(G^4_{-2/3}(t)+1)} } \alpha(t,r_{\ast}),
	\end{aligned}
\end{equation*}
and 
\begin{equation*}
	Q_{r_{\ast}}=\frac{\gamma(0)}{\gamma(r_{\ast})}=e^{\frac{\delta_2 r_{\ast}}{2}}.
\end{equation*}
By \eqref{6lem11} and $\alpha(\cdot,\cdot) \in \widehat{\mathfrak{I}}$, it follows that
\begin{equation}\label{6lem12}
	\rho_{r_{\ast}}(t) \leq \hat{\rho}_{r_{\ast}}(t) \leq \hat{\rho}_{r_{\ast}}(T_0) , \quad \forall t \leq T_0,
\end{equation}
where
\begin{equation*}
	\hat{\rho}_{r_{\ast}}(t)=C(1+e^{C(G^4_{-2/3}(t)+1)}) \alpha(t,r_{\ast}),
\end{equation*}
for some $C>0$ large enough.

From the invariant property of the pullback $\mathcal{D}$-attractor $\mathcal{A}=\{A(t):t\in \mathbb{R}\}$ and recall that $A(t) \subseteq B(t) (\forall t \in \mathbb{R})$, it deduces that
\begin{equation}\label{6lem13}
	A(t)=U(t,t-r)A(t-r) \subseteq U(t,t-r)B(t-r), \quad \forall t \in \mathbb{R}, \ \forall r \geq 0.
\end{equation}
Then, combining \eqref{6lem12} and \eqref{6lem13}, Lemma \ref{lem28} yields for each $t \leq T_0$, 
\begin{equation*}
\begin{aligned}
	&dist_{\mathbb{K}^1}(A(t),\overline{\mathbf{B}}_{\mathbb{K}^{4/3}}(\hat{\rho}_{r_{\ast}}(t))) \\
	&\leq dist_{\mathbb{K}^1}(U(t,t-r)D(t-r),\overline{\mathbf{B}}_{\mathbb{K}^{4/3}}(\hat{\rho}_{r_{\ast}}(t)))\\
	&\leq Q_{r_{\ast}}(\gamma(r_{\ast}))^{\frac{r}{r_{\ast}}}R_0(t-r)\\
	&\leq Ce^{-\frac{\delta_2(r-r_{\ast})}{2}}R_0(t-r)\\
	&\leq C	e^{-\frac{\delta_2(r-r_{\ast})}{2}}(G_{-1}(t-r)+1) \to0, \quad \text{as} \ r \to +\infty.
\end{aligned}
\end{equation*}
Thus
\begin{equation*}
	A(t) \subseteq \overline{\mathbf{B}}_{\mathbb{K}^{4/3}}(\hat{\rho}_{r_{\ast}}(t)) \subseteq \overline{\mathbf{B}}_{\mathbb{K}^{4/3}}(\hat{\rho}_{r_{\ast}}(T_0)), \quad \forall t \leq T_0.
\end{equation*}
The proof of the lemma is thus finished.
\end{proof}
\begin{Theorem}\label{Th42}
Let assumptions $\left(\mathbf{H1}\right)-\left(\mathbf{H3}\right)$ be satisfied. Assume $g \in L^2_{loc}(\mathbb{R}; \boldsymbol{E}^{0})$ and \eqref{5lem11} holds true. Then\\
$i)$ If $g(t)$ satisfies
	\begin{equation}\label{7lem11}
		\widehat{G}(t):=\left(\int_{-\infty}^{t} e^{\sigma_3s} \|g(s)\|^2 {\rm d}s \right)^{\frac{1}{2}} <+\infty, \quad  \forall t \in \mathbb{R},
	\end{equation}
where $0<\sigma_3 \leq \delta_3$, then, the pullback $\mathcal{D}$-attractor $\mathcal{A} =\{A(t):t \in \mathbb{R}\} $ given by Theorem \ref{Th36} satisfies that
	\begin{equation}\label{7lem09}
		A(t) \ \text{is bounded in}\ \mathbb{K}^{2} \text{, for any }\ t \in \mathbb{R}.
	\end{equation}
$ii)$ If $g(t)$ satisfies
\begin{equation}\label{7lem12}
	G(t):=\sup_{\tau \leq t} \left( \int_{\tau -1}^{\tau} \|g(s)\|^2 {\rm d}\right)^{\frac{1}{2}} \leq +\infty, \quad \forall t\in \mathbb{R},
\end{equation}
then,
	\begin{equation}\label{7lem10}
		\bigcup_{s \leq t} A(s) \ \text{is bounded in }\ \mathbb{K}^{2} \text{, for any }\ t \in \mathbb{R}.
	\end{equation}
\end{Theorem}
\begin{proof}
	We split the solution $U(t,\tau)x_{\tau}=(u(t),u_t)$ with $x_{\tau} \in A(\tau)$ into the sum
	\begin{equation*}
		U(t,\tau)x_{\tau}=Y(t,\tau)x_{\tau}+Z(t,\tau)x_{\tau},
	\end{equation*}
	where $Y(t,\tau)x_{\tau}=(v(t),0)$ and $Z(t,\tau)x_{\tau}=(w(t),w_t)=(w(t),u_t)$ solve the following probelms respectively
	\begin{equation}\label{7ine01}
		\begin{cases}
			\partial_t v +A \partial_t v + A v =0 & in \  (\tau,\infty)\times\Omega,\\
			v(t,x)=0 & on \ (\tau,\infty)\times\partial\Omega,\\
			v (\tau,x)=v_{\tau}=u^0&\tau \in \mathbb{R},x \in \Omega,\\
			v(\tau+\theta,x)=0 & \theta \in(-h,0),x \in \Omega,
		\end{cases}
	\end{equation}
	and 
	\begin{equation}\label{7ine02}
		\begin{cases}
			\partial_t w +A \partial_t w + A w =\Psi & in \  (\tau,\infty)\times\Omega,\\
			w(t,x)=0 & on \ (\tau,\infty)\times\partial\Omega,\\
			w (\tau,x)=0&\tau \in \mathbb{R},x \in \Omega,\\
			w(\tau+\theta,x)=\phi(\theta,x)  & \theta \in(-h,0),x \in \Omega,
		\end{cases}
	\end{equation}
	where
	\begin{equation*}
		\Psi(t)=g(t)+b(t,w_t)-f(u).
	\end{equation*}
Applying the same argument as equations \eqref{4lem01} and \eqref{6lem10}, we can estimate the solution  of \eqref{7ine01} as follows
\begin{equation}\label{7lem05}
	\| Y(t,\tau) x_{\tau} \|_{\mathbb{K}^1} \leq C e^{-\frac{\delta_1}{2}(t-\tau)} \| v_{\tau} \|_1.
\end{equation}
Multiplying the first equation in \eqref{7ine02} by $w$ and integrating it over $\Omega$, we have
	\begin{equation}
	\frac{1}{2} \frac{d}{dt}\left( \| w \|^2 +\|  w \|^2_1 \right)+ \|  \hat{w} \|^2_1 \leq \langle \hat{\Psi}(t), w \rangle.
\end{equation}
In light of \eqref{1eq01} and the following Sobolev embedding
\begin{equation*}
	\boldsymbol{E}^1 \hookrightarrow L^6(\Omega),
\end{equation*}
we deduce that 
\begin{equation*}
\begin{aligned}
	&| \langle \Psi(t), w \rangle|  \\
	&\leq |\langle b(t,w_t),w \rangle |+| \langle g(t), w \rangle |+| \langle f(u), w \rangle |\\
	&\leq \left\| g(t) \right\|_{-1}\left\|w \right\|_1+\left\| b(t,w_t) \right\|\left\| w \right\|  + C\int_{\Omega}(1+\left| u \right|^5 ) \left| w\right| {\rm d}x  \\
	& \leq \left\| g(t) \right\|^2_{-1}+\frac{1}{4} \left\| w \right\|^2_1+C\left\| b(t,w_t) \right\|^2 +\frac{1}{2} \left\| w \right\|^2 +C(1+\left\| u \right\|^5_{L^6})+\frac{1}{4} \left\| w \right\|^2_{L^6} \\
	&\leq \left\| g(t) \right\|^2_{-1}+\left\| b(t,w_t) \right\|^2+C\left( \left\| u \right\|^{10}_{1}+1\right)+\left( \frac{1}{2}+\frac{1}{2\lambda_1}\right)\left\| w \right\|^2_1.
\end{aligned}
\end{equation*}
Similar to \eqref{5lem07}, we deduce that
\begin{equation}
	\int_{t-h}^{t}\| \nabla w(s) \|^2 {\rm d}s \leq C\left( t-\tau +1\right) \left( G^{30}_{-1}(t)+1\right).
\end{equation}
Multiplying the first equation in \eqref{7ine02} by $Aw$ and integrating it over $\Omega$, we have
\begin{equation}\label{7lem01}
	\frac{1}{2} \frac{d}{dt}(\| w(t) \|^2_1+\|w(t)\|^2_2)+\|w \|^2_2 = \langle \Psi(t) ,Aw \rangle.
\end{equation}
Using \eqref{1eq01} and the Sobolev embeddings $\boldsymbol{E}^{4/3} \hookrightarrow \boldsymbol{E}^{6/5} \hookrightarrow L^{10}(\Omega)$, we find that
	\begin{align}\label{7lem02}
		| \langle \Psi(t) ,Aw \rangle | &\leq | \langle g(t), Aw \rangle |+ | \langle b(t,w_t), Aw \rangle |+ | \langle f(u),Aw \rangle | \nonumber \\
		&\leq \| g(t)\| \|w\|_2+ \|b(t,w_t) \| \| w\|_2+C \int_{\Omega} (1+|u|^5) |Aw| {\rm d}x \nonumber \\
		&\leq  \| g(t)\| \|w\|_2+ \|b(t,w_t) \| \| w\|_2+C(1+\|u\|^5_{L^{10}}) \|w\|_2 \nonumber \\
		& \leq \frac{1}{2}\|w\|^2_2+C(\|u\|^{10}_{4/3}+\|b(t,w_t)\|^2+\|g(t)\|^2+1).
	\end{align}  
Let $T_0 \in \mathbb{R}$ be  arbitrary fixed. By the invariant property of $\mathcal{A}=\{A(t):t\in \mathbb{R}\}$ and $x_{\tau}=(u^0,\phi) \in A(\tau)$, using Lemma \ref{L41}, we have
\begin{equation}\label{7lem03}
	(u(t),u_t) \in A(t) \subseteq \overline{\mathbf{B}}_{\mathbb{K}^{4/3}}(\hat{\rho}_{r_{\ast}}(T_0)), \quad \forall t \leq T_0.
\end{equation}
From \eqref{7lem01}-\eqref{7lem02}, there exists constant $\delta_3$ such that
\begin{equation}\label{7lem08}
	\frac{d}{dt}(\| w(t) \|^2_1+\|w(t)\|^2_2)+\delta_3 (\| w(t) \|^2_1+\|w(t)\|^2_2) \leq C(\hat{\rho}^{10}_{r_{\ast}}(T_0)+\|b(t,w_t)\|^2+\|g(t)\|^2+1).
\end{equation}
Multiplying \eqref{7lem08} by $e^{\sigma_3 t}$, consequently,
\begin{equation*}
	\begin{aligned}
		&\frac{d}{dt}e^{\sigma_3 t}(\| w(t) \|^2_1+\|w(t)\|^2_2) \\
		&\leq -(\delta_3-\sigma_3)e^{\sigma_3 t}(\| w(t) \|^2_1+\|w(t)\|^2_2)+Ce^{\sigma_3 t}(\hat{\rho}^{10}_{r_{\ast}}(T_0)+\|b(t,w_t)\|^2+\|g(t)\|^2+1),
	\end{aligned}
\end{equation*}
where $\sigma_3 \leq \delta_3$.\\
Integrating the above inequality from $\tau$ to $t$ and noting that $w(\tau)=0$, it follows that
\begin{equation}\label{7lem04}
		\| w(t) \|^2_1+\|w(t)\|^2_2\leq C e^{-\sigma_3t} \int_{\tau}^{t} e^{\sigma_3 s}(\hat{\rho}^{10}_{r_{\ast}}(T_0)+\|b(t,w_t)\|^2+\|g(t)\|^2+1) {\rm d}s .
\end{equation}
If $g(t)$ satisfies \eqref{7lem11}, (\uppercase \expandafter{\romannumeral 2}) and (\uppercase \expandafter{\romannumeral 5}), then \eqref{7lem04} yields
\begin{equation}
	\| w(t) \|^2_1+\|w(t)\|^2_2 \leq C(\hat{\rho}^{10}_{r_{\ast}}(T_0)+e^{-\sigma_3 t}\widehat{G}^2(t) +\int_{\tau -h}^{t} \|w(s)\|^2 {\rm d}s),
\end{equation}
which implies that 
\begin{equation}\label{7lem06}
	\| Z(t,\tau)x_{\tau} \|_2 \leq \rho_1(t), \quad \forall \tau \leq t \leq T_0,
\end{equation}
where $\rho_1(t)=C(\hat{\rho}^{10}_{r_{\ast}}(T_0)+e^{-\sigma_3 t}\widehat{G}^2(t) +\int_{\tau -h}^{t} \|w(s)\|^2 {\rm d}s)$.

For any $t \leq T_0$, combining \eqref{7lem05} and \eqref{7lem06}, we obtain
\begin{equation*}
	\begin{aligned}
		&dist_{\mathbb{K}^1}(A(t),\overline{\mathbf{B}}_{\mathbb{K}^2}(\rho_1(t))) \\
		&= dist_{\mathbb{K}^1}(U(t,t-r)A(t-r),\overline{\mathbf{B}}_{\mathbb{K}^2}(\rho_1(t))) \\
		&\leq Ce^{-\delta_1 r}\| v_{\tau} \|_1 \to 0, \quad \text{as} \ r\to +\infty.
	\end{aligned}
\end{equation*}
Hence 
\begin{equation*}
	A(t) \subseteq \overline{\mathbf{B}}_{\mathbb{K}^2}(\rho_1(t)), \quad \forall t \leq T_0,
\end{equation*}
which implies \eqref{7lem09}.

On the other hand, if $g(t)$ satisfies \eqref{7lem12}, from \eqref{7lem04}, we infer
\begin{equation}
	\begin{aligned}
		&\| w(t) \|^2_1+\|w(t)\|^2_2 \\
		&\leq C e^{-\eta t} \int_{\tau}^{t} e^{\eta s}(\hat{\rho}^{10}_{r_{\ast}}(T_0)+\|b(t,w_t)\|^2+\|g(t)\|^2+1) {\rm d}s \\
		& \leq C(\hat{\rho}^{10}_{r_{\ast}}(T_0)+G^2(t)+\int_{\tau -h}^{t} \|w(s)\|^2 {\rm d}s).
	\end{aligned}
\end{equation}
Consequently,
\begin{equation}\label{7lem07}
		\| Z(t,\tau)x_{\tau} \|_2 \leq \rho_2(T_0), \quad \forall \tau \leq t \leq T_0,
\end{equation}
where $\rho_2(T_0)=C(\hat{\rho}^{10}_{r_{\ast}}(T_0)+G^2(t)+\int_{\tau -h}^{t} \|w(s)\|^2 {\rm d}s)$.

Owing to \eqref{7lem05} and \eqref{7lem07}, for every $t \leq T_0$, we conclude that
\begin{equation}
	\begin{aligned}
		&dist_{\mathbb{K}^1}(A(t),\overline{\mathbf{B}}_{\mathbb{K}^2}(\rho_2(T_0))) \\
		&= dist_{\mathbb{K}^1}(U(t,t-r)A(t-r),\overline{\mathbf{B}}_{\mathbb{K}^2}(\rho_2(T_0))) \\
		&\leq Ce^{-\delta_1 r}\| v_{\tau} \|_1 \to 0, \quad \text{as} \ r\to +\infty.
	\end{aligned}
\end{equation}
Hence 
\begin{equation}
	A(t) \subseteq \overline{\mathbf{B}}_{\mathbb{K}^2}(\rho_2(t)), \quad \forall t \leq T_0,
\end{equation}
which proves \eqref{7lem10}.
\end{proof}

\section{Declarations}
\textbf{Ethical Approval}

Not applicable.

\textbf{Competing interests}

We declare that this work is original and has not been published previously. No conflict of interes exists in the submission of this manuscript, and it is approved by all authors for publication.

\textbf{Authors' contributions}

Equal contributions by all authors. All authors read and approved the final manuscript.

\textbf{Funding}

Qin, Cai and Wang were supported by the National Natural Science Foundation of China with contract number 12171082, the fundamental research funds for the central universities with contract numbers 2232022G-13,2232023G-13 and by a grant from science and technology commission of Shanghai municipality. Mei was supported  by NSERC Grant RGPIN 2022-03374. 

\textbf{Availability of data and materials }

Data sharing not applicable to this article as no datasets were generated or analysed during the current study.

\end{document}